\documentclass[10pt,a4paper,reqno]{amsart}

\usepackage[foot]{amsaddr}
\title{Embedded Trefftz discontinuous Galerkin methods}%

\author{Christoph Lehrenfeld}
\author{Paul Stocker}
\address{Georg-August-Universit\"at, G\"ottingen, Germany}
\email{lehrenfeld@math.uni-goettingen.de}
\email{p.stocker@math.uni-goettingen.de}

\usepackage[T1]{fontenc}
\usepackage{a4wide}
\usepackage{amsmath,amsthm,amsfonts,amssymb,bm}
\usepackage{mathtools}
\usepackage{enumitem}\setitemize{leftmargin=5mm}
\usepackage{accents}
\usepackage{color}
\usepackage{booktabs}
\usepackage[usenames,dvipsnames,svgnames,table]{xcolor}
\usepackage{cite}
\usepackage{soul}
\usepackage{hyperref}
\usepackage[capitalize]{cleveref}

\hypersetup{
   colorlinks=true,
   urlcolor=green!70!black,
   citecolor=blue!70!black,
   linkcolor=green!70!black
}

\usepackage{graphicx}
\usepackage{fancyhdr}
\usepackage{todonotes}
\usepackage{pgfplotstable}

\theoremstyle{plain}
\newtheorem{theorem}{Theorem}
\newtheorem{lemma}[theorem]{Lemma}

\theoremstyle{definition}

\newtheorem{remark}[theorem]{Remark}

\newcommand{\dofs}{\texttt{dofs}}
\newcommand{\ndofs}{\texttt{ndofs}}
\newcommand{\nnzes}{\texttt{nnzes}}

\newcommand{\inner}[1]{\langle #1 \rangle}

\newcommand\restr[2]{{ \left.\kern-\nulldelimiterspace #1 \vphantom{\big|} \right|_{#2} }}
\newcommand{\ee}{{\rm e}}
\DeclareMathOperator{\Div}{\mathrm{div}}
\DeclareMathOperator{\Id}{\mathrm{id}}

\DeclareMathOperator\Ker{ker}
\DeclareMathOperator\Range{range}
\DeclareMathOperator\dt{\frac{\partial}{\partial t}}
\newcommand{\calO}{\mathcal{O}}
\newcommand{\calG}{\mathcal{G}}
\newcommand{\calL}{\mathcal{L}}

\newcommand{\sst}{\;\text{s.t.}\;}

\newcommand{\IT}{\mathbb{T}}

\newcommand{\IN}{\mathbb{N}}

\newcommand{\IP}{\mathbb{P}}
\newcommand{\Th}{{\mathcal{T}_h}} 
\newcommand{\Fh}{\mathcal{F}_h} 

\newcommand{\dom}{\Omega} 
\newcommand{\nx}{n_\mathbf{x}} 

\newcommand{\Vhp}{V^p(\Th)}
\newcommand{\Vp}{V^p(K)}
\newcommand{\be}{\mathbf{e}}
\newcommand{\bQ}{\mathbf{Q}}
\newcommand{\bR}{\mathbf{R}}
\newcommand{\bI}{\mathbf{I}}
\newcommand{\bT}{\mathbf{T}}
\newcommand{\bW}{\mathbf{W}}
\newcommand{\bl}{\mathbf{l}}
\newcommand{\bM}{\mathbf{M}}

\newcommand{\bA}{\mathbf{A}}
\newcommand{\bU}{\mathbf{U}}
\newcommand{\bV}{\mathbf{V}}
\newcommand{\bu}{\mathbf{u}}

\newcommand{\bw}{\mathbf{w}}
\newcommand{\bb}{\mathbf{b}}
\newcommand{\bSig}{\mathbf{\Sigma}}
\newcommand{\bx}{{\mathbf x}}
\newcommand{\bj}{{\mathbf j}}
\newcommand{\bq}{{\mathbf q}}

\newcommand{\jump}[1]{[\![ #1 ]\!]}

\newcommand{\avg}[1]{\{\!\!\{#1\}\!\!\}}

\newcommand{\QT}{{\mathbb{Q\!T}}}
\newcommand{\mi}{{\boldsymbol{i}}}

\newcommand{\wt}{{(w,\bm \tau)}}

\newcommand{\DGmethod}{DG}
\newcommand{\ETmethod}{emb. Trefftz}

\newcommand{\Tmethod}{$\IT^p$}
\newcommand{\QTmethod}{$\QT^p$}
\newcommand{\HDGmethod}{HDG}

\usepackage{stmaryrd} 
\usepackage{pgf,tikz}
\usetikzlibrary{spy,matrix, calc, arrows,shapes,decorations}
\usepackage{tkz-euclide}
\usetikzlibrary{positioning,arrows,calc,decorations.markings,math,arrows.meta}
\usepackage{pgfplots} 
\usepgfplotslibrary{groupplots}
\usepackage{csvsimple} 
\usepackage{siunitx} 
\usepackage{tabularx}
\usepackage{booktabs}
\usepackage{multirow}
\usepackage[normalem]{ulem}
\useunder{\uline}{\ul}{}

\pgfplotscreateplotcyclelist{paulcolors4}{%
violet,every mark/.append style={solid,fill=violet},mark=square*,very thick,mark size=3pt\\%
teal,every mark/.append style={solid,fill=teal},mark=*,very thick,mark size=3pt\\%
orange,every mark/.append style={solid,fill=orange},mark=diamond*,very thick,mark size=3pt\\%
cyan,every mark/.append style={solid,fill=cyan},mark=star,very thick,mark size=3pt\\%
}

\pgfplotscreateplotcyclelist{paulcolors}{%
violet,every mark/.append style={solid,fill=violet},mark=square*,very thick,mark size=3pt\\%
teal,every mark/.append style={solid,fill=teal},mark=*,very thick,mark size=2.8pt\\%
orange,every mark/.append style={solid,fill=orange},mark=diamond*,very thick,mark size=2pt\\%
violet,densely dashed,every mark/.append style={solid,fill=violet},mark=square*,very thick,mark size=3pt\\%
teal,densely dashed,every mark/.append style={solid,fill=teal},mark=*,very thick,mark size=2.8pt\\%
orange,densely dashed,every mark/.append style={solid,fill=orange},mark=diamond*,very thick,mark size=2pt\\%
violet,dash dot,every mark/.append style={solid,fill=violet},mark=square*,very thick,mark size=3pt\\%
teal,dash dot,every mark/.append style={solid,fill=teal},mark=*,very thick,mark size=2.8pt\\%
orange,dash dot,every mark/.append style={solid,fill=orange},mark=diamond*,very thick,mark size=2pt\\%
}
\pgfplotscreateplotcyclelist{paulcolors2}{%
teal,every mark/.append style={solid,fill=teal},mark=*,very thick,mark size=3pt\\%
orange,every mark/.append style={solid,fill=orange},mark=diamond*,very thick,mark size=2.5pt\\%
teal,densely dashed,every mark/.append style={solid,fill=teal},mark=*,very thick,mark size=3pt\\%
orange,densely dashed,every mark/.append style={solid,fill=orange},mark=diamond*,very thick,mark size=2.5pt\\%
teal,dash dot,every mark/.append style={solid,fill=teal},mark=*,very thick,mark size=3pt\\%
orange,dash dot,every mark/.append style={solid,fill=orange},mark=diamond*,very thick,mark size=2.5pt\\%
}
\pgfplotscreateplotcyclelist{paulcolors1}{%
teal,every mark/.append style={solid,fill=teal},mark=*,very thick,mark size=3pt\\%
}
\pgfplotscreateplotcyclelist{paulcolorsfill}{%
teal,fill=teal,fill opacity=0.5,thick\\
orange,fill=orange,fill opacity=0.5,thick\\
violet,fill=violet,fill opacity=0.5,thick\\
}

\pgfplotsset{
    discard if not/.style 2 args={
        x filter/.append code={
            \edef\tempa{\thisrow{#1}}
            \edef\tempb{#2}
            \ifx\tempa\tempb
            \else
                
            \fi
        }
    }
}
\pgfplotsset{compat=1.16}

\pgfplotsset{tick label style={font=\small},label style={font=\small},legend style={font=\small},}
\pgfplotsset{ width=.49\linewidth}

\newcommand{\logLogSlopeTriangle}[5]
{

    \pgfplotsextra
    {
        \pgfkeysgetvalue{/pgfplots/xmin}{\xmin}
        \pgfkeysgetvalue{/pgfplots/xmax}{\xmax}
        \pgfkeysgetvalue{/pgfplots/ymin}{\ymin}
        \pgfkeysgetvalue{/pgfplots/ymax}{\ymax}

        \pgfmathsetmacro{\xArel}{#1}
        \pgfmathsetmacro{\yArel}{#3}
        \pgfmathsetmacro{\xBrel}{#1-#2}
        \pgfmathsetmacro{\yBrel}{\yArel}
        \pgfmathsetmacro{\xCrel}{\xArel}

        \pgfmathsetmacro{\lnxB}{\xmin*(1-(#1-#2))+\xmax*(#1-#2)} 
        \pgfmathsetmacro{\lnxA}{\xmin*(1-#1)+\xmax*#1} 
        \pgfmathsetmacro{\lnyA}{\ymin*(1-#3)+\ymax*#3} 
        \pgfmathsetmacro{\lnyC}{\lnyA-#4*(\lnxA-\lnxB)}
        \pgfmathsetmacro{\yCrel}{\lnyC-\ymin)/(\ymax-\ymin)} 

        \coordinate (A) at (rel axis cs:\xArel,\yArel);
        \coordinate (B) at (rel axis cs:\xBrel,\yBrel);
        \coordinate (C) at (rel axis cs:\xCrel,\yCrel);

        \draw[#5]   (A)-- node[pos=0.5,anchor=north] {}
                    (B)-- 
                    (C)-- node[pos=0.5,anchor=west] {#4}
                    cycle;
    }
}

\usepackage{listings}
\usepackage[plain]{algorithm}
\usepackage{algpseudocodex}

\definecolor{maroon}{cmyk}{0, 0.87, 0.68, 0.32}
\definecolor{halfgray}{gray}{0.55}
\definecolor{ipython_frame}{RGB}{207, 207, 207}
\definecolor{ipython_bg}{RGB}{247, 247, 247}
\definecolor{ipython_red}{RGB}{186, 33, 33}
\definecolor{ipython_green}{RGB}{0, 128, 0}
\definecolor{ipython_cyan}{RGB}{64, 128, 128}
\definecolor{ipython_purple}{RGB}{170, 34, 255}

\lstdefinelanguage{iPython}{
    morekeywords=[2]{abs,all,any,basestring,bin,bool,bytearray,callable,chr,classmethod,cmp,compile,complex,delattr,dict,dir,divmod,enumerate,eval,execfile,file,filter,float,format,frozenset,getattr,globals,hasattr,hash,help,hex,id,input,int,isinstance,issubclass,iter,len,list,locals,long,map,max,memoryview,min,next,object,oct,open,ord,pow,property,range,raw_input,reduce,reload,repr,reversed,round,set,setattr,slice,sorted,staticmethod,str,sum,super,tuple,type,unichr,unicode,vars,xrange,zip,apply,buffer,coerce,intern},%
    sensitive=true,%
    morecomment=[l]\#,%
    morestring=[b]',%
    morestring=[b]",%
    morestring=[s]{'''}{'''},
    morestring=[s]{"""}{"""},
    morestring=[s]{r'}{'},
    morestring=[s]{r"}{"},%
    morestring=[s]{r'''}{'''},%
    morestring=[s]{r"""}{"""},%
    morestring=[s]{u'}{'},
    morestring=[s]{u"}{"},%
    morestring=[s]{u'''}{'''},%
    morestring=[s]{u"""}{"""},%
    literate=
    {á}{{\'a}}1 {é}{{\'e}}1 {í}{{\'i}}1 {ó}{{\'o}}1 {ú}{{\'u}}1
    {Á}{{\'A}}1 {É}{{\'E}}1 {Í}{{\'I}}1 {Ó}{{\'O}}1 {Ú}{{\'U}}1
    {à}{{\`a}}1 {è}{{\`e}}1 {ì}{{\`i}}1 {ò}{{\`o}}1 {ù}{{\`u}}1
    {À}{{\`A}}1 {È}{{\'E}}1 {Ì}{{\`I}}1 {Ò}{{\`O}}1 {Ù}{{\`U}}1
    {ä}{{\"a}}1 {ë}{{\"e}}1 {ï}{{\"i}}1 {ö}{{\"o}}1 {ü}{{\"u}}1
    {Ä}{{\"A}}1 {Ë}{{\"E}}1 {Ï}{{\"I}}1 {Ö}{{\"O}}1 {Ü}{{\"U}}1
    {â}{{\^a}}1 {ê}{{\^e}}1 {î}{{\^i}}1 {ô}{{\^o}}1 {û}{{\^u}}1
    {Â}{{\^A}}1 {Ê}{{\^E}}1 {Î}{{\^I}}1 {Ô}{{\^O}}1 {Û}{{\^U}}1
    {œ}{{\oe}}1 {Œ}{{\OE}}1 {æ}{{\ae}}1 {Æ}{{\AE}}1 {ß}{{\ss}}1
    {ç}{{\c c}}1 {Ç}{{\c C}}1 {ø}{{\o}}1 {å}{{\r a}}1 {Å}{{\r A}}1
    {€}{{\EUR}}1 {£}{{\pounds}}1
    {^}{{{\color{ipython_purple}\^{}}}}1
    {=}{{{\color{ipython_purple}=}}}1
    {+}{{{\color{ipython_purple}+}}}1
    {*}{{{\color{ipython_purple}$^\ast$}}}1
    {/}{{{\color{ipython_purple}/}}}1
    {+=}{{{+=}}}1
    {-=}{{{-=}}}1
    {*=}{{{$^\ast$=}}}1
    {/=}{{{/=}}}1,
    literate=
    *{-}{{{\color{ipython_purple}-}}}1
     {?}{{{\color{ipython_purple}?}}}1,
    identifierstyle=\color{black}\ttfamily,
    commentstyle=\color{ipython_cyan}\ttfamily,
    stringstyle=\color{ipython_red}\ttfamily,
    keepspaces=true,
    showspaces=false,
    showstringspaces=false,
    rulecolor=\color{ipython_frame},
    framexleftmargin=0mm,
    numbers=left,
    numberstyle=\tiny\color{halfgray},
    numbersep=1mm,
    xleftmargin=1mm,
    basicstyle=\scriptsize,
    keywordstyle=\color{ipython_green}\ttfamily,
}

\lstdefinestyle{trefftzy}{
    language=iPython,
    emptylines=1,
    breaklines=true,
    basicstyle=\footnotesize\ttfamily\color{black},    
    moredelim=**[is][\color{teal}]{<}{>},
    moredelim=**[is][\color{purple}]{'}{'},
}

\definecolor{pscol}{rgb}{0.8,0,0}

\definecolor{rvcol}{rgb}{1,0,0}
\newcommand{\rv}[1]{{#1}}

\begin{document}
\sloppy
\begin{abstract}
In Trefftz discontinuous Galerkin methods a partial differential equation is discretized using discontinuous shape functions that are chosen to be elementwise in the kernel of the corresponding differential operator. We propose a new variant, the embedded Trefftz discontinuous Galerkin method, which 
is the Galerkin projection of an underlying discontinuous Galerkin method onto a subspace of Trefftz-type.
The subspace can be described in a very general way and to obtain it no Trefftz functions have to be calculated explicitly, instead the corresponding embedding operator is constructed.
In the simplest cases the method recovers established Trefftz discontinuous Galerkin methods. But the approach allows to conveniently extend to general cases, including inhomogeneous sources and non-constant coefficient differential operators. 
We introduce the method, discuss implementational aspects and explore its potential on a set of standard PDE problems. Compared to standard discontinuous Galerkin methods we observe a severe reduction of the globally coupled unknowns in all considered cases, reducing the corresponding computing time significantly. Moreover, for the Helmholtz problem we even observe an improved accuracy similar to Trefftz discontinuous Galerkin methods based on plane waves.
\end{abstract}



\maketitle

\section{Introduction}
In this manuscript we propose a novel numerical method closely related to Trefftz discontinuous Galerkin methods.
The main idea of Trefftz methods, originating from \cite{trefftz1926}, is to choose optimal discretization spaces that provide the same approximation quality as comparable discrete spaces with a significant reduction of the number of degrees of freedom (\ndofs). 
For an overview on Trefftz methods see \cite{LLHC08,Qin05,kk95}.

Polynomial Trefftz functions have been obtained for several linear partial differential operators with constant coefficients such as Laplace equation \cite{MiWi55a,MiWi55b}, acoustic wave equation \cite{zbMATH06677366,mope18,zbMATH02190564}, heat equation \cite{RoWi}, plate vibration and beam vibration equation \cite{macikag2011trefftz, al2008method}, and time-dependent Maxwell's equation \cite{kretzschmarphd}.
Efforts to generate Trefftz polynomials in a general case have been undertaken, see \cite{MiWi56,Ho}.
The idea of finding Trefftz polynomials via Taylor series, used in \cite{macig2005solution}, was extended in \cite{yang2020trefftz} to construct Trefftz-like polynomials in the case of non-linear elliptic PDEs possibly with smooth coefficients.
The method requires a Taylor expansion of the coefficients of the PDE to construct these polynomials via a recursive procedure for the coefficients of a Taylor polynomial. 

Discontinuous Galerkin (DG) schemes can be easily combined with Trefftz functions as the basis construction of elements decouples. 
Trefftz-DG schemes for Laplace equations have been analyzed in \cite{HMPS14,LiShu2006,LiShu2012}.
Different time-dependent problems have been discussed recently, such as 
wave problems in one space dimension \cite{SpaceTimeTDG,KSTW2014,PFT09},
the acoustic wave equation \cite{StockerSchoeberl,mope18,bgl2016},
elasto-acoustics \cite{bcds20},
a class of Friedrichs systems coming from linear transport \cite{buet2020trefftz,morel2018trefftz},
for time-dependent Maxwell's equation see \cite{kretzschmarphd,EKSW15},
and the linear Schr\"odinger equation \cite{tdgschroedinger}. 
Very popular applications of Trefftz methods are wave propagation problems in frequency domain. There, no polynomial Trefftz space exists and plane wave functions are used instead.
A DG scheme for Helmholtz equation has been presented and analyzed in \cite{cessenat1998application,ghp09,mps13,hmp11b,AndreaPhD}, for more information on Trefftz methods for Helmholtz equation see the recent survey \cite{TrefftzSurvey} and the references therein.
Plane waves have also been used for linear elasticity \cite{moiola2013plane} and time-harmonic Maxwell's equation \cite{HMP11,fure2020discontinuous,AndreaPhD}.
\rv{Apart from combining Trefftz spaces with DG schemes, harmonic polynomials and plane waves have also been combined with virtual element methods, see \cite{zbMATH06596743,zbMATH06996162,zbMATH07139237}}.

In the case of a differential operator with smooth coefficients the local solutions in a Trefftz space are not sufficient to provide high-order approximation. 
This can be circumvented by weakening the requirements of the Trefftz space, providing again a sufficiently large basis.
In \cite{qtrefftz} a quasi-Trefftz DG method is introduced and analyzed for the acoustic wave equation with smooth coefficients.
Plane waves have been generalized to work with a DG method for the Helmholtz equation with smooth coefficient in \cite{ImbertGeradDespres2014}.

Trefftz methods usually are applied to homogeneous problems (problems with no volume source term).
To treat inhomogeneous problems with a Trefftz approach, one needs to construct a particular solution. Then it is possible to apply Trefftz methods to solve for the difference.
In \cite{zbMATH02139293} the Poisson problem is treated in two separate ways, once a fundamental solutions are used as trial functions, while for the other approach radial basis functions are used to express the inhomogeneous part.
Another approach for Helmholtz and time-harmonic Maxwell's equation is presented in \cite{HuYuan18}, where a series of inhomogeneous local problems are solved using spectral elements to obtain local particular solutions.

\subsection{Main contributions and outline of the paper}
The embedded Trefftz method proposed here side-steps the explicit construction of a Trefftz space, instead we present a simple method of constructing a corresponding embedding. 
In its simplest form, the method proposed here can be seen as a convenient way to set up the linear system of a Trefftz DG discretization by means of a Galerkin projection of a standard DG method onto its Trefftz subspace. 
Instead of implementing Trefftz functions explicitly, the characterization of the Trefftz function space as the kernel of an associated differential operator is exploited to construct an embedding of a Trefftz subspace into the DG space in a very generic way. We denote this embedding as the \emph{Trefftz embedding}.
The construction requires only element-local operations of small matrices and is embarrassingly parallel. To compute the discrete kernel numerically several established methods exists, e.g. one can use QR factorizations, singular, or eigenvalue decompositions. We denote this approach as \emph{embedded Trefftz DG method}.

When the embedded Trefftz DG approach is used to implement an existing Trefftz DG method one trades the convenient setup for slightly larger costs in the setup of the linear system. 
However, the strongest feature of the embedded Trefftz DG approach is the following:
The generic way the linear systems are set up allows to apply the concept of Trefftz methods beyond their previous limitations.  Two of these limitations that can be conveniently exceeded are:
\begin{itemize}
\item In many interesting applications a standard Trefftz space is either not polynomial or unfeasible to set up at all limiting the application of Trefftz DG methods for these cases.
\item If a standard Trefftz DG subspace exists, still only homogeneous equations (equations with no volume source terms) are easily dealt with. Inhomogeneous equations are difficult to deal with for standard Trefftz DG methods.
\end{itemize}


In cases where a polynomial Trefftz space is not embedded in the discretization space or when a Trefftz basis is unknown, the proposed embedded Trefftz DG method can still be applied.
This is done by slightly weakening the properties of the Trefftz space we are looking to embed.
With this we are able to treat for example the Helmholtz equation and PDEs with piecewise smooth coefficients, e.g. the acoustic wave equation.

The embedded Trefftz DG method offers a more generic way of obtaining local particular solutions
exploiting the underlying DG discretization. 
Computing element-wise particular solutions in the DG space can be done with little computational effort and then be used to homogenize the DG system.

The paper continuous as follows: In \cref{sec:method} we present the method, recovering traditional Trefftz DG methods and its extensions to the case of weaker Trefftz spaces and inhomogeneous equations.
In \cref{sec:impl} we address implementational aspects and turn to applications of the approach to a set of PDE problems in \cref{sec:num}.
We present numerical examples for Laplace equation, Poisson equation, acoustic wave equation with piecewise constant and also with smooth coefficient, Helmholtz equation, and a linear transport equation.
\rv{In \cref{sec:compareschemes} we compare the Trefftz DG method to another popular acceleration technique for DG methods, the Hybrid DG method.}
The method was implemented using \texttt{NGSolve} and \texttt{NGSTrefftz}\footnote{see \url{https://ngsolve.org/} and \url{https://github.com/PaulSt/NGSTrefftz}.}.
Conclusion and outlook are presented in \cref{sec:outlook}.

\section{The method}\label{sec:method}
We introduce the method in several steps. After some preliminaries, we first introduce the basic concepts of the method in \cref{sec:method1}.
Here, we will assume that a suitable Trefftz space exists, and is embedded in the discrete DG space $\Vhp$. Here, we understand as the Trefftz space the subspace of functions that locally fulfill the PDE pointwise. 
This case only applies to a small set of problems where (polynomial) Trefftz spaces exist, like the Laplace equation and the acoustic wave equation with piece-wise constant wavespeed. 
We then turn to cases where a Trefftz space is not embedded in the discrete space. In this case we consider a weak Trefftz space instead. This is presented in \cref{sec:method2}. 
How inhomogeneous equations can easily be dealt with is explained in \cref{sec:method3} resulting in a method that allows to treat a large class of PDE problems in a uniform way. 

\subsection{Preliminaries}\label{sec:prelim}
We consider a linear partial differential operator $\calL$ and a corresponding boundary (and possibly initial) value problem 
\begin{equation} \label{eq:basicpde}
   \calL u = f \text{ in } \Omega \subset \mathbb{R}^d
\end{equation}
supplemented by suitable boundary conditions and
paired with a suitable discontinuous Galerkin formulation 
\begin{equation}\label{eq:basicdg}
    \text{Find }u_{hp}\in \Vhp,~\text{ s.t. }
    a_h(u_{hp},v_{hp})=\ell(v_{hp})\qquad \forall v_{hp}\in \Vhp.
\end{equation}
The DG formulation is set on a mesh $\Th$ of the domain $\dom$, and we assume that the space $\Vhp$ is the product of local spaces $\Vp$ for $K\in\Th$ where the index $p$ indicates a polynomial degree, e.g. $\Vp = \mathcal{P}^p(K)$ the space of polynomials up to degree $p$. 
We note that space-time DG formulations are allowed in this setting, so that $d\in \mathbb{N}$ is either the dimension of the space or space-time domain.

\subsection{Embedded Trefftz method}\label{sec:method1}
At first, in this subsection, we make the following simplifying assumptions:
Firstly, we assume $f=0$. Second, we assume that the differential operator has the form $\calL = \sum_{l=1}^d \alpha_l D_l^{\beta_l}$
for $\alpha_l \in \mathbb{R}$ 
and $\beta_l \in \mathbb{N}$ and that all mesh elements are straight. Valid examples are $\calL = -\Delta$, $\calL = b \cdot \nabla$ for $b \in \mathbb{R}^d$, $\calL = \partial_t + b \cdot \nabla$, but not $\calL = -\Delta \pm \Id$, $\calL = -\Delta + b \cdot \nabla$ or $\calL = -\Div( \alpha \nabla \cdot)$ for a non-constant field $\alpha$. This assumption ensures that the subspace of $\Vp$ that we define next is sufficiently large to allow for reasonable approximations of the solution.   

We define the Trefftz space 
\begin{align} \label{eq:trefftzspace}
    \IT^p(\Th)=\{v\in \Vhp,\ \calL v=0 \text{ on each } K\in\Th\}\subset \Vhp.
\end{align}
The Trefftz version of \eqref{eq:basicdg} is its Galerkin projection to $\IT^p(\Th)$:
\begin{equation}\label{eq:trefftzdg}
    \text{Find }u_{\IT}\in \IT^p(\Th),~\text{ s.t. }
    a_h(u_{\IT},v_{\IT})=\ell(v_{\IT})\qquad \forall v_{\IT}\in \IT^p(\Th).
\end{equation}

In contrast to previous works on Trefftz methods, we do not construct the space $\IT^p$ from scratch based on accordingly defined basis function, but rather aim to build an embedding operator for the Trefftz space into the underlying piecewise polynomial space $\Vp$. We call this the \emph{embedded Trefftz DG method}.

\subsubsection*{Construction of the embedding operator}

Let $\{\phi_i\}$ be the set of basis functions of $\Vhp$, $N =\operatorname{dim}(\Vhp)$ and $\calG: \mathbb{R}^n \to \Vhp$ be the Galerkin isomorphism, $\calG(\bx) = \sum_{i=1}^N \bx_i \phi_i$. With $\be_i,~i=1,..,N$ the canonical unit vectors in $\mathbb{R}^n$, so that $\calG(\be_i) = \phi_i$, we define the following matrices and vector for $i,j = 1,\dots, N$  
\begin{subequations}
\begin{align} \label{def:matrices}
    (\bA)_{ij}&=a_h(\calG(\be_j),\calG(\be_i))=a_h(\phi_j,\phi_i),
     \qquad 
    (\bl)_i = \ell(\calG(\be_i)) = \ell(\phi_i) \\
    (\bW)_{ij}&=\inner{\calL\phi_j,\calL\phi_i}_{0,h}, \label{def:W}
\end{align}
\end{subequations}
where $\inner{\cdot,\cdot}_{0,h}=\sum_{K\in\Th}\inner{\cdot,\cdot}_K$ is the element-wise $L^2$-inner product.
We are interested in the kernel of $\calL$ (in an element-wise and pointwise sense) in $\Vhp$ as this is the part where the Trefftz DG method operates. We note that 
\begin{equation}
    \Ker(\calL) = \calG(\Ker(\bW))
\end{equation}    
and hence are looking for a basis of $\Ker(\bW)$. 
As $\bW$ is block diagonal, with blocks corresponding to the elements $K \in \Th$ we will construct the basis element by element which implies that the following calculations can be done in parallel.

For a fixed $K \in \Th$ let $\bW_K \in \mathbb{R}^{N_{\!K}\!\times\! N_{\!K}}$ be the corresponding block with $N_K = \dim(\Vp)$. 
For the dimension of the kernel we have $M_K := \dim(\Ker(\calL)) = N_K - L_K$ with $L_K = \dim(\Range(\calL))$ where (for the case of differential operator $\calL$ in the assumed form) $L_K$ is typically known.
Now, we can determine the kernel of $\bW_K$ and collect a set of orthogonal basis vectors in a matrix  $\bT_K\in\mathbb{R}^{N_{\!K}\!\times\! M_{\!K}}$ so that $\ker(\bW_K) = \bT_K \cdot \mathbb{R}^{M_K}$. The kernel matrix $\bT_K$ can be determined  numerically, e.g. by a QR decomposition or a singular value decomposition, we discuss this in more depth in \cref{app:determinekernel}.

Composing the individual element matrices $\bT_K$ into a block matrix $\bT\in\mathbb{R}^{N \times M}$ with $M = \dim(\Ker(\bW)) = \sum_{T\in\Th} M_T$
we easily obtain the characterization of the kernel as $\Ker(\bW) = \bT \cdot \mathbb{R}^{M}$. Further note that there holds $\bT^T \bT = \bI_{M \times M}$.

We define the Galerkin isomorphism between $\mathbb{R}^M$ and $\IT^p(\Th)$ as $\calG_\IT: \mathbb{R}^M \to \IT^p(\Th)$, $\bx \mapsto \calG(\bT \bx)$ and denote a discrete operator corresponding to $\bT$ as the \emph{Trefftz embedding} $T: \IT^p(\Th) \to \Vhp, v_h \mapsto \calG (\bT \calG_{\IT}^{-1}(v_h))$.

To setup the linear systems corresponding to \eqref{eq:trefftzdg} we first assemble $\bA$, $\bl$ and $\bT$ and arrive at: Find $\bu_\IT (=\calG_\IT^{-1}(u_\IT))$ so that 
\begin{equation} \label{eq:trefftzlinearsys}
   \bT^T\bA\bT ~ \bu_\IT = \bT^T \bl.
\end{equation}

The resulting linear system is well controlled in terms of its conditioning.
\begin{lemma}[Conditioning of the embedded Trefftz method]\label{lem:cond}
    The spectral relative condition number of the embedded Trefftz DG method is bounded by the corresponding condition of the DG method, 
   \begin{equation}
    \kappa_2(\bT^T \bA \bT) \leq  \kappa_2(\bA).
   \end{equation}
\end{lemma}
\begin{proof}
    By construction of $\bT$ all its column vectors are orthogonal.     
\end{proof}

\begin{remark}[Matrix representing the kernel of $\calL$]
There are other possibilities to characterize the kernel of $\calL$ in $\Vp$. Here, we chose the specific form as it most obviously displays that 
$\bu\in\Ker{\bW}$ is equivalent to $\calL \calG (\bu) = 0$ pointwise as there has to hold $\inner{\calL \calG (\bu), \calL \calG (\bu) }_{0,h}=0$. 
In \cref{sec:method2} we will replace $\bW$ with a proper generalization. 
\end{remark}
\begin{remark}[Degrees of freedom]
    One major advantage of the Trefftz method is the reduction of the number of (globally coupled) degrees of freedom (\ndofs) without harming the approximation too much. Let us first discuss the reduction on an example. Assume $\calL$ is a second order operator, say $-\Delta$, the unknown field is scalar and the mesh is triangular. Then $N_K = {(p+1)(p+2)}/{2}$ and $-\Delta \Vp = V^{p-2}(K)$ and hence $L_K = {(p-1)p}/{2}$ yields $M_K = 2p+1$. This holds for every element and similar calculations can be made for different differential operators and vectorial problems. 
    In general, the Trefftz method decreases \ndofs~from $\mathcal{O}(p^d)$ to $\mathcal{O}(p^{d-1})$ and is hence especially attractive for higher order methods. In the lowest order case $p=0$ or also $p=1$ (depending on the differential operator) one often has $\IT^p(\Th) = \Vhp$.
    A similar reduction of \ndofs~from $\mathcal{O}(p^d)$ to $\mathcal{O}(p^{d-1})$ in a discontinuous Galerkin setting is also achieved in methods that associate \dofs~primarily to facets and rely in some way on static condensation such as Hybrid DG methods \cite{cockburn2016static,cockburn2009unified} or Hybrid High order (HHO) methods \cite{di2016review}. We discuss the comparison between these approaches in more detail in \cref{sec:compareschemes}.
\end{remark}
\begin{remark}[Uniqueness of the Trefftz embedding]
    We notice that while the embedding $T:\IT^p(\Th)\to V^p(\Th)$ is unique the matrix $\bT$ and hence, the basis for $\IT^p(\Th)$ resulting from the previous procedure is not. 
\end{remark}
\begin{remark}[Volume integrals] \label{rem:volints}  
    One advantage of Trefftz DG methods that is sometimes advertised is that volume integrals that involve $\calL u$ (possibly after partial integration) can be removed during assembly as $\calL u = 0$ on the Trefftz space. 
This is also possible when using the embedded Trefftz method, see also \cref{rem:qtrefftzterms}.
\end{remark}

\subsection{Weak Trefftz embedding}\label{sec:method2}
In the previous section we restricted the differential operator to have a special form, here we aim to lift some of the restrictions, specifically we want to treat non-constant coefficients and mixed differential orders also including $\Id$.
In these situations the Trefftz space as defined in \eqref{eq:trefftzspace} will typically much too small, i.e. the requirement $\calL v = 0$ \emph{on the polynomial spaces} may restrict the discrete solution too much leading to a version of \emph{locking}.

To circumvent this problem, we weaken our condition in the Trefftz space. We introduce a projection $\Pi$ that is yet to be defined and define the \emph{weak Trefftz space} and the \emph{embedded weak Trefftz DG method}:
\begin{subequations}
\begin{align}
    \text{Find }u_{\IT}\in \IT^p(\Th)&,~\text{ s.t. }
    a_h(u_{\IT},v_{\IT})=\ell(v_{\IT})\qquad \forall v_{\IT}\in \IT^p(\Th)\quad \text{ with } \\
    \IT^p(\Th)&=\{v\in \Vhp,\ \Pi \calL v=0\text{ on each } K\in\Th\}. \label{eq:weakTspace}
\end{align}
\end{subequations} 
This replaces the pointwise condition $\calL v = 0$ with something that is potentially weaker. 
The projection $\Pi$ will be designed to ensure that the resulting weak Trefftz space is reasonably large to allow for proper approximations of the PDE solution. 

A projection that has proven fruitful is given by 
\begin{equation}\label{eq:weakTproj}
    \Pi: L^2(\dom) \to V^\bq(\Th), \text{ s.t. } \inner{\Pi w, v}_{0,h} = \inner{w, v}_{0,h} ~ \forall~ v \in V^\bq(\Th).
\end{equation}
The multi-index $\bq\in\IN^d$ on $V^\bq$ denotes the polynomial space of maximal order $\bq=(q_1,\dots,q_d)$ in each variable $(x_1,\dots,x_d)$.
(If all entries of the multi-index are equal, we will continue to write $V^q$.)
Let us now consider the differential operator $\calL = \sum_{\bj\in\IN^d} \alpha_\bj D^{\bj}$ with $\alpha_\bj \in C(\Th)$.
Then we choose $q_i$ equal to the largest appearing differential order, i.e. $q_i=p-\max_{\alpha_\bj\not\equiv 0} j_i,\ i=1,\dots,d$. 

With these definitions, the analogue to $\bW$ from \cref{sec:method1}, i.e. the matrix that can be used to numerically extract the weak Trefftz space, becomes for $j=1,...,N$ and $i=1,...,\tilde{N} =\dim V^{\bq}( \mathcal{T}_h )$ and $\{\psi_i\}$ a basis for $V^{\bq}( \mathcal{T}_h )$
\begin{equation} \label{def:W2} 
    (\bW)_{ij} =\inner{\calL \phi_j,\psi_i}.
\end{equation}
In \cref{sec:implW} we present a slightly more convenient version (with the same kernel) that we also used in the numerical examples in this manuscript.

\begin{remark}[Consistency with \cref{sec:method1}]
    We recover the Trefftz method from \cref{sec:method1} if we choose $V^\bq=\calL V^{p}$.
    For instance for $\calL = -\Delta$ we can choose $V^\bq=\Delta V^p=V^{p-2}$.
    \rv{From the constant coefficient case treated in \cref{sec:method1}, it becomes clear that smaller values in $\bq$ would make the space larger than the Trefftz space, resulting in a sub-optimal reduction of the space.}
\end{remark}

\begin{remark}[Quasi-Trefftz]  \label{rem:qtrefftz}
    In \cite{qtrefftz} weak Trefftz spaces have been used to treat the acoustic wave equation with smooth coefficients, we recall the quasi-Trefftz spaces used there in \eqref{eq:QT}.
    We recover this quasi-Trefftz method if we replace $\Pi$ in \eqref{eq:weakTspace} with a Taylor polynomial expansion of order $p-1$ in the element center. 
\end{remark}

\subsection{Inhomogeneous PDEs} \label{sec:method3}
The Trefftz methods discussed in \cref{sec:method2} and \cref{sec:method1} as well as those known in the literature are designed for the homogeneous problem $\calL u = 0$ only and hence are not able to solve for problems with non-zero source term. 
A simple adjustment however allows to deal also with the inhomogeneous case for the embedded Trefftz DG method. 

In this section we treat equations of the form
\begin{align*}
   &\calL u = f \text{ in } \Omega
\end{align*}
for $f\in L^2(\Omega)$,  
supplemented by suitable boundary conditions
for which a suitable DG discretizations is given in  the form 
\begin{equation}\label{eq:basicdgwrhs}
    \text{Find }u_{hp}\in \Vhp,~\text{ s.t. }
    a_h(u_{hp},v_{hp})=\ell(v_{hp}) = g(v_{hp})+\inner{f,v_{hp}}_{0,h}   \qquad \forall v_{hp}\in \Vhp.
\end{equation}
where the linear form $g(\cdot)$ corresponds to the weak imposition of boundary conditions.

On the continuous level of the PDE we can introduce an affine shift and decompose the solution to $\calL u = f$ with $u = u_0 + u_f$.
Here $u_f$ is a particular solution to $\calL u_f = f$ and $u_0$ (uniquely) solves $\calL u_0 = 0$ (with boundary data depending on $u_f$).
As the embedded Trefftz (or weak Trefftz) method is based on an underlying DG discretization we are able to construct a reasonable particular solution in the DG space $\Vhp$ and apply the same homogenization strategy. 

We then write
$u_h = u_{\IT}  + u_{h,f}$ for $u_{\IT}\in\IT^p(\Th)$ and $u_{h,f}\in\Vhp$. 
$u_{h,f}$ is a (non-unique) particular solution to $\Pi \calL u_{h,f} = \Pi f$ that we can compute in an element-local fashion, see \cref{app:partsolutions}.
For the existence of a particular solution we require that $\Pi \calL:\Vhp \to V^\bq (\Th )$ is surjective. 
This essentially depends on the choice of $\calL$, $p$ and $\bq$.
It always holds true in the setting of \cref{sec:method1}. In the remainder we assume that a particular solution exists. 

For $u_{h,f}$ a particular solution, we are looking for a solution $u_h\in\IT^p(\Th) + u_{h,f}$ so that
\begin{equation} \label{eq:inhom}
    a_h(u_{h},v_{\IT})
 =\ell(v_{\IT}) ~~ \forall~ v_{\IT}\in\IT^p(\Th).
\end{equation}
After homogenization this means that we are looking for
$u_{\IT}\in\IT^p(\Th)$ that (uniquely) solves 
\begin{equation}
    a_h(u_{\IT},v_{\IT})
 =\ell(v_{\IT}) - a_h(u_{h,f},v_{\IT}) ~~ \forall~ v_{\IT}\in\IT^p(\Th).
\end{equation}
This translates to the solution of the linear system 
\begin{align}
    \bT^T\bA\bT \bu_{\IT}  = \bT^T (\bl-\bA\bu_f).
\end{align}
\begin{remark}[Degenerate cases] 
    Let us briefly discuss two extreme cases of the Trefftz embedding. If $\mathcal{L} =\operatorname{id}$, then $\IT^p(\Th) = \{0\}$ and the discrete solution of $\mathcal{L} u = u = f$ will effectively be computed when determining the - in this case unique - particular solution $u_{h,f}$. The other extreme case is $\mathcal{L} = 0$ for which we recover $\IT^p(\Th) =\Vhp$. 
\end{remark}
Assuming a coercive problem and a DG formulation that provides good discretization properties, on the embedded (weak) Trefftz DG subspace we can still apply all essential analysis tools to obtain a C\'ea-type result: 
\begin{lemma}\label{lem:cea}
    Let $u\in V(\dom)$ be a weak solution to the PDE problem under consideration and $u_h \in \IT^p(\Th) + u_{h,f}$ be the solution to \eqref{eq:inhom} and $\| \cdot \|_{h}$ be a suitable norm on $\Vhp\cup V(\dom)$. 
    Assume that problem \eqref{eq:basicdg} is well-posed, specifically $a_h$ and $\ell$ are continuous with respect to $\|\cdot \|_{h}$ on $\Vhp\cup V(\dom)$ and $a_h$ is coercive (with respect to $\|\cdot \|_{h}$) on $\Vhp$.
Then 
\begin{align*}
    \|u-u_{h}\|_{h}& \lesssim \inf_{\substack{v_h\in\Vhp\\ \Pi\calL v_h=\Pi f}} \|u-v_h\|_{h} 
\end{align*}
\end{lemma}
\begin{proof}
    Let $v_h\in\IT^p(\Th)+u_{h,f}=\Vhp\cap\{v\sst \Pi\calL v=\Pi f\}$, and $w_h=u_{h}-v_h\in\IT^p(\Th)$. Then there holds \vspace*{-0.5cm}
\begin{align*}
    \|u_h-v_{h}\|_{h}^2 \lesssim &\ a_h(u-v_h,w_h) + \overbrace{a_h(u_h-u,w_h)}^{=0}\\
                          \lesssim &\ \|u-v_h\|_{h}\|u_h-v_h\|_{h} 
\end{align*}
where we used coercivity and continuity and that $a_h(u-u_h,v)=0\ \forall v\in\IT^p(\Th)$.
Hence $\|u_h-v_h\|_{h}\lesssim \|u-v_h\|_{h}$. 
Adding a triangle inequality and taking the infimum yields the result.
\end{proof}

The crucial question is of course on the approximation quality of the (affinely shifted) weak Trefftz space. We leave that open for future research.

\section{Implementational aspects}\label{sec:impl}
Let us consider several implementational aspects in this separate section. 
First, 
in \cref{sec:implW} we discuss how to set up a version of $\bW$ - the matrix that characterizes the (weak) Trefftz subspace - 
in a more convenient way than what directly follows from \cref{sec:method2}. 
Next, we discuss how the dimension of the (weak) Trefftz subspace can be obtained in an implementation.
Afterwards, we present the algorithmic structure of the embedded Trefftz DG method in comparison to a standard DG method.
We discuss how to numerically compute the kernel of the matrix $\bW$ and the particular solutions in \cref{app:determinekernel,app:partsolutions}. 
In \cref{sec:complexity} we discuss the computational complexity of the embedded Trefftz DG method.

\subsection{Characterization and implementation of the (weak) Trefftz DG space} 
\label{sec:implW}  
Involving a new space $V^\bq (\mathcal{T}_h)$ in the implementation and assembling of the matrix $\bW_K$ as used in \eqref{def:W2} may occur somewhat unneccessarily cumbersome. In the following, we propose a different but equivalent formulation that is more convenient and is also applied in the numerical examples later. 
Instead of using $V^\bq (\mathcal{T}_h)$ we again use the finite element space $\Vhp$, and find a differential operator $\tilde\calL$ with $\tilde\calL: \Vhp \to V^\bq (\mathcal{T}_h)$ surjective.
The operator $\tilde\calL$ is a modification of $\calL$ which only keeps the highest order derivatives in each direction and has only constant coefficients. In the simplest cases, i.e. in the setting of \cref{sec:method1} we simply choose $\tilde\calL = \calL$. Otherwise, in the general setting of \cref{sec:method2}, for example for the Helmholtz equation with variable wavespeed, with the operator $\calL=-\Id-\omega(\bx)\Delta$, we choose $\tilde\calL=\Delta$ (which would correspond to $\bq_i = p-2$). 

This way, we can re-define the matrix $\bW$ from \eqref{def:W2} for the general case as 
\begin{align} \label{def:W3}
    (\bW)_{ij}&=\inner{\calL\phi_j,\tilde\calL\phi_i}_{0,h}. 
\end{align}
Note that this matrix $\bW$ has dimension $N\times N$ as in the original definition \eqref{def:W} in the setting of \cref{sec:method1}.

\subsection{Dimension of the (weak) Trefftz subspace} \label{sec:dimofk}
On each element $K\in \mathcal{T}_h$ the dimension of the $\ker \Pi\calL$, $M_K$, can be computed as $N_K - \dim (\Range{\Pi\calL} )$ also for the (weak) Trefftz spaces. In most cases this expression could easily be precomputed and be given to the algorithm for the numerical kernel extraction. Alternative - and this is a bit more convenient as it requires less user input - we can also try to read that information from the matrix directly. Next, we explain how this can be achieved and notice that this approach has been used in all numerical examples below.

When numerically computing the kernel of a matrix, we use a QR, singular value or eigenvalue decomposition. In all these factorizations the diagonal of one of these factors (e.g. the diagonal matrix of the SVD or the triangular matrix of the QR decomposition) would have $M_K$ zeros assuming exact arithmetics. 
Due to inexact computer arithmetics this will not be the case exactly and hence we use a truncation parameter $\varepsilon>0$ to determine which values are considered as (numerical) zeros.
A numerical study on the choice of the threshold for the Laplace problem is done in \cref{sec:lap}.

\subsection{The algorithm} \label{sec:algo}
In \cref{alg:et} we show how the embedded Trefftz DG method is implemented. We display the method in pseudo-code and a python code example in \texttt{NGSolve}, using the operator $\tilde\calL$ as described in \cref{sec:implW}. 
Let us briefly comment on the algorithm. 

1. First of all, we notice that the computation and the application of the Trefftz embedding is separated from the setup of the linear system of the DG method. 
On the one hand this obviously leaves room for (efficiency) improvements as an integration into the DG assembly could allow to reduce the need for storing the DG matrices (and vectors in the homogeneous case) in the first place. 
On the other hand it also facilitates the realization of the embedded Trefftz DG method in a DG software framework without the need to interfere with optimized routines such as the assembly.
\rv{We focus on the second approach in the presentation of the implementation in \cref{alg:et}.
Both approaches are implemented in \texttt{NGSolve}+\texttt{NGSTrefftz} and can be found in the software documentation, see \cite{Stocker2022}.}

2. Further, we observe that the setup of the matrix $\bT$ is done element-by-element and can hence be carried out in an embarrassingly parallel manner. The resulting matrix is block diagonal such that the setup of $\bT^T\bA\bT$ and the r.h.s. vector can also be done efficiently. 

3. Assuming a corresponding DG method already exists for a user, he only needs to additionally specify the operators $\calL$ and $\tilde\calL$,the truncation parameter $\varepsilon$ and the r.h.s. to obtain the corresponding embedded Trefftz DG method. All of these inputs are canonical in most cases rendering the approach and its implementation quite easily accessible.

\begin{algorithm}
\noindent\begin{minipage}{.43\textwidth}
    \begin{algorithmic}[1]
        \small
        \Require Basis functions $\{\phi_i\}_i$, 
        DG formulation ($a_h$, $l$),
        {\color{teal} operators $\calL$, $\tilde{\calL}$,
        truncation parameter $\varepsilon$,
        {\color{purple} r.h.s. $f$} 
        } 
        \Function{dg matrix}{}
            \State $(\bA)_{ij}=a_h(\phi_j,\phi_i)$
            \State $(\bl)_i = \ell(\phi_i)$
        \EndFunction
        {\color{teal}
        \For{$K\in\Th$}
            \State $(\bW_K)_{ij}=\inner{\calL\phi_j,\tilde \calL\phi_i}_{0,h}$
            \State $\bT_K=\ker_h(\varepsilon;\bW_K)$
            \If{$f\neq 0$}
            \State ${\color{purple}(\bw_K)_{i}=\inner{f,\tilde \calL\phi_i}_{0,h}} $ 
            \State ${\color{purple}(\bu_{f})_K=\bW^\dagger_K \bw_K}$
            \EndIf
        \EndFor
        } 
        \State Solve ${\color{teal}\bT^T}\bA{\color{teal}\bT} ~ \bu_\IT = {\color{teal}\bT^T}  {\color{purple}(}\bl {\color{purple} - \bA \bu_f)} $
        \State $\bu_h={\color{teal}\bT} \bu_\IT {\color{purple} + \bu_f}$
        \State \Output $\bu_h$
    \end{algorithmic}
\end{minipage}
\hfill
\begin{minipage}{.52\textwidth}
\begin{lstlisting}[language=iPython,style=trefftzy]
def Solve(mesh, order, dgscheme<,>
          <L, Ltilde, eps>','
          'rhs'):
  fes = L2(mesh,order=order,dgjumps=True)
  uh = GridFunction(fes)
  a,f = dgscheme(fes)
  u,v = fes.TnT()
  <W = L(u)*Ltilde(v)*dx>
  'w = rhs*Ltilde(v)*dx'
  <T', uf' = TrefftzEmbedding(W,fes,eps',w')>
  <Tt = T.CreateTranspose()>
  TA = <Tt@>a.mat<@T>
  ut = TA.Inverse()<*(Tt>*'('f.vec'-a.mat*uf))'
  uh.vec.data = <T*>ut '+ uf'
  return uh
\end{lstlisting}
\end{minipage}
\caption{Embedded Trefftz DG algorithm in the version discussed in \cref{sec:impl}. Pseudo-code (left) and running python code (using \texttt{NGSolve} and \texttt{NGSTrefftz}, right). 
The parts used by a standard DG scheme are colored black, the parts using the embedded Trefftz DG method, as in \cref{sec:method2}, are colored in teal. Code needed to treat the inhomogeneous problem, presented in \cref{sec:method3}, is colored purple.
}\label{alg:et}
\end{algorithm}

\subsection{Computing the kernel matrix \texorpdfstring{$\bT_K$}{}}\label{app:determinekernel}   
In this section we briefly discuss how to numerically extract the kernel matrix $\bT_K$ from $\bW_K^T$ so that $\Ker(\bW_K^T) = \bT_K \cdot \mathbb{R}^{M_K}$. 
\rv{ The elementwise matrix $\bW_K$ is of size $\tilde N_K\times N_K$, where $\tilde N_K$ depends on the approach: 
If the matrix $\bW_K$ stems from either \eqref{def:W} or \eqref{def:W3}, then $\bW_K$ is a square matrix and $\tilde N_K=N_K$. If \eqref{def:W2} is used, then $\tilde N_K =\dim V^{\bq}(K)$.}
We discuss two possibilities: One based on a QR decomposition and one a singular value decomposition.
\subsubsection*{QR decomposition}
The most obvious option seems to be a QR decomposition of $\bW_K^T$, s.t. $ \bW_K^T=\bQ_K \cdot \bR_K$ for an orthogonal matrix $\bQ_K$, of size ${N_K\!\times\!N_K}$, and an upper triangular matrix $\bR_K$, sized ${N_K\times \tilde N_K}$, where we assume an ordering such that \( (\bR_K)_{ij} = 0\) for $i = L_K+1,..,\tilde N_K$. Then we have $\bW_K = \bR_K^T \cdot \bQ_K^T$ and that the last $M_K$ columns of $\bQ_K$ span the kernel of $\bW_K$. We denote this submatrix as $\bT_K$, i.e. $(\bT_K)_{i,j} := (\bQ_K)_{i,L_K+j}$ for $i=1,..,N_K$, $j=1,..,M_K$. One easily checks that there holds $\ker(\bW_K) = \bT_K \cdot \mathbb{R}^{M_K}$.
\subsubsection*{Singular value decomposition}
Applying a singular value decomposition (SVD) to $\bW_K$ yields $\bW_K = \bU_K \bSig_K \bV_K^T$ for two orthogonal matrices $\bU_K$, $\bV_K$ and a diagonal matrix $\bSig_K$ with $(\bSig_K)_{jj} = 0$ for $j>L_K$. 
Then one easily sees that the last $M_K$ columns of $\bV_K$ span the kernel of $\bW_K$ and we choose this submatrix as $\bT_K$, i.e. $(\bT_K)_{i,j} := (\bV_K)_{i,L_K+j}$ for $i=1,..,N_K$, $j=1,..,M_K$.

\subsection{Computing particular DG solutions}\label{app:partsolutions}  
To compute a (local) particular solution needed to solve inhomogeneous PDEs, see \cref{sec:method3}, on an element $K\in \mathcal{T}_h$. We assemble $(\bw_K)_{i}=(f,\calG(\be_i))=(f,\tilde\calL\phi_i)$ and define $(\bu_{f})_K=\bW^\dagger_K \bw_K.$
Here $\bW^\dagger_K$ denotes the pseudoinverse of the matrix $\bW_K$, which can be obtained using the QR decomposition or SVD of the matrix $\bW_K$ which may have already been computed when numerically computing the kernel of $\bW_K$, cf. \cref{app:determinekernel}.

\subsection{Algorithmic complexity} \label{sec:complexity}
The computation of the kernel of $\Pi \calL$ or a particular solution scales well in the mesh size and can easily be parallelized. However, the scaling in the local unknowns per element is severe. For a discretization of order $p$ the unknowns on each element scale like $\mathcal{O}(p^d)$. Hence, the costs for element-wise QR or SV decompositions (including the setup of the adjusted matrices and vectors and neglecting parallelization effects) are $\mathcal{O}(h^{-d} p^{3d})$. Let us compare these costs with the solver costs of standard DG and classical Trefftz DG methods. For a standard DG method we have $\mathcal{O}(h^{-d}p^d)$ unknowns globally and assuming a solver complexity $\mathcal{O}(\texttt{N}^\alpha)$, where \texttt{N} denotes the \ndofs~with $\alpha \in \mathbb{R}$, the costs are $\mathcal{O}( (h^{-d}p^{d})^\alpha)$. For a classical Trefftz DG method the costs are $\mathcal{O}((h^{-d}p^{(d-1)})^\alpha)$. The proposed method has a bad asymptotic complexity if $h = const$ and $p \to \infty$. In this case the computational costs will soon be dominated by the QR or SVD operations in the Trefftz DG setup. However, in the opposite case, $p = const$ and $h \to 0$ asymptotically the costs of the embedded Trefftz setup are only relevant for an optimal solver complexity $\alpha = 1$ and are negligible otherwise. In the applications below we observe that at least moderately high (fixed) polynomial degrees can be easily used with the embedded Trefftz DG method without dominating the computational costs by the computation of the embedding. 

\section{Applications}\label{sec:num}
We consider a set of different PDE problems and aim to compare the embedded Trefftz DG method to the underlying DG method (with full polynomial basis) and to a standard Trefftz basis, whenever available. 
We provide the scripts used to obtain the numerical results presented in this section\footnote{for reproduction material see \url{https://doi.org/10.25625/JIO1MP}} 
and interactive code examples to try out the implementation online without any prerequisites\footnote{find the documentation at \url{https://github.com/PaulSt/NGSTrefftz}}.

In \cref{sec:lap,sec:wave} we consider the Laplace equation and the acoustic wave equation, where explicit polynomial Trefftz basis functions are available for comparison.
For the Helmholtz equation, treated in \cref{sec:helm}, an explicit Trefftz basis is available, however it is not polynomial.
In \cref{sec:poi} we consider Laplace equation with a source term, applying the techniques described in \cref{sec:method3}.
In \cref{sec:qwave,sec:adv} we consider PDEs with non-constant coefficients.
In \cref{sec:compareschemes} we present a comparison to HDG methods in terms of computational costs associated to the linear solvers.

\subsection{Preliminaries and Notation}
Although there exist many DG schemes for each of the problems considered here, in order to focus on the comparison of the different bases, we will stick to one DG scheme per equation.

Further, we concentrated on simplicial meshes here, but the approach carries directly over to hexahedral and even polygonal meshes immediately.
If not explicitly mentioned we do not exploit the possibility to remove volume integrals, cf. \cref{rem:volints}.

In the following numerical experiments we will always use sparse direct solvers, i.e. \texttt{umfpack} \cite{10.1145/992200.992206} for non-symmetric and a \texttt{sparsecholesky} solver implemented in \texttt{NGSolve} for symmetric linear systems. 

In all examples, for the embedded Trefftz DG method we used the local SVD to compute the numerical kernel and particular soluations. 

For the purpose of brevity in the plots below we label the standard DG method with \DGmethod and the corresponding Trefftz method with \Tmethod.

For the description of the DG schemes in the following we introduce some  standard DG notation. We denote by $\Fh$ the set of facets and distinguish
$\Fh^\text{int}$, the set of interior facets, from $\Fh^\text{bnd}$, the set of boundary facets. 
Let $K$ and $K'$ be two neighboring elements sharing a facet $F \in \Fh^\text{int}$. On $F$ the functions $u_{K}$ and $u_{K'}$ denote the two limits of a discrete function from the different sides of the element interfaces. $n_K$ and $n_{K'}$ are the unit outer normals to $K$ and $K'$. We define
\begin{equation*}
    \jump{v} := v_{K} \cdot n_K + v_{K'} \cdot n_{K'}, \qquad
    \avg{v} := \frac12 v_{K} + \frac12 v_{K'}.
\end{equation*}
On the boundary facets we set $\jump{v} = v_K \nx$ and $\avg{v} = v$ where 
$\nx$ denotes the (spatial) outer normal on the boundary.

\subsection{Laplace equation}\label{sec:lap}
We start with the Laplace equation with Dirichlet boundary conditions
\begin{align*}
    \begin{cases}
    -\Delta u = 0 &\text{ in } \dom, \\
    u=g &\text{ on } \partial \dom,
    \end{cases}
\end{align*}
and consider a symmetric IP-DG discretization, cf. \cite{arnold2002unified}, given by
\begin{align}\label{eq:dglap}
    \begin{split}
    a_h(u,v) &= \int_\dom \nabla u\nabla v\ dV
    -\int_{\Fh^\text{int}}\left(\avg{\nabla u}\jump{v}+\avg{\nabla v}\jump{u} 
    - \frac{\alpha p^2}{h}\jump{u}\jump{v} \right) dS \\
           &\qquad -\int_{\Fh^\text{bnd}}\left(\nx\cdot\nabla u v+\nx\cdot\nabla v u-\frac{\alpha p^2}{h} u v \right) dS\\
    \ell(v) &= \int_{\Fh^\text{bnd}}\left(\frac{\alpha p^2}{h} gv -\nx\cdot\nabla vg\right) dS.
    \end{split}
\end{align}
As interior penalty parameter we choose $\alpha=4$.
We apply our method as described in \cref{sec:method1} with $\calL=-\Delta$.
The (embedded) Trefftz DG space, as in \eqref{eq:trefftzspace}, is now the space of harmonic polynomials.
For the numerical example we set the boundary condition $g$ such that the exact solution is given by
\begin{align}
    & u=\exp(x)\sin(y) && \text{on } \dom=(0,1)^2,\label{eq:lap2d}\\
    & u=\exp(x+y)\sin(\sqrt{2}z) && \text{on }\dom=(0,1)^3,\label{eq:lap3d}
\end{align}
and consider nested simplicial unstructured meshes created by refining a coarse simplicial unstructured mesh of initial mesh size $h\approx 0.5$.

\begin{figure}[ht]
    \hspace*{-0.02\textwidth}\resizebox{1.025\linewidth}{!}{
    \begin{tikzpicture}
        \begin{groupplot}[%
          group style={%
            group name={my plots},
            group size=3 by 1,
            horizontal sep=1.5cm,
          },
        legend style={
            legend columns=3,
            at={(-0.8,-0.2)},
            anchor=north,
            draw=none
        },
        xlabel={$p$},
        ymajorgrids=true,
        grid style=dashed,
        cycle list name=paulcolors
        ]      
        \nextgroupplot[ymode=log, ylabel={DG-error}]
        \foreach \h in {0.25}{
            \addplot+[discard if not={h}{\h}] table [x=p, y=error, col sep=comma] {results/lap2d8.csv};
            \addplot+[discard if not={h}{\h}] table [x=p, y=terror, col sep=comma] {results/lap2d8.csv};
            \addplot+[discard if not={h}{\h}] table [x=p, y=svdterror, col sep=comma] {results/lap2d8.csv};
        }
        \nextgroupplot[ymode=log, xlabel={\ndofs},ylabel={DG-error}]
        \foreach \h in {0.25}{
            \addplot+[discard if not={h}{\h}] table [x=ndof, y=error, col sep=comma] {results/lap2d8.csv};
            \addplot+[discard if not={h}{\h}] table [x=tndof, y=terror, col sep=comma] {results/lap2d8.csv};
            \addplot+[discard if not={h}{\h}] table [x=tndof, y=svdterror, col sep=comma] {results/lap2d8.csv};
        }
        \nextgroupplot[ymode=log,ylabel={Condition number}]
        \foreach \h in {0.25}{
            \addplot+[discard if not={h}{\h}] table [x=p, y=l2cond, col sep=comma] {results/lap2d16.csv};
            \addplot+[discard if not={h}{\h}] table [x=p, y=tcond, col sep=comma] {results/lap2d16.csv};
            \addplot+[discard if not={h}{\h}] table [x=p, y=svdtcond, col sep=comma] {results/lap2d16.csv};
        }
        \legend{\DGmethod, \Tmethod, \ETmethod}
        \end{groupplot}
    \end{tikzpicture}}
    \vspace*{-0.5cm}
    \caption{Results for the Laplace problem in 2 dimensions with exact solution \eqref{eq:lap2d} on a fixed mesh with $h=0.25$ for different values of polynomial order $p$.
    Left: Convergence in terms of polynomial order $p$. 
    Center: Convergence in terms of degrees of freedom of the linear system.
    Right: Condition number of the discrete system.
}
    \label{fig:lap}
\end{figure}
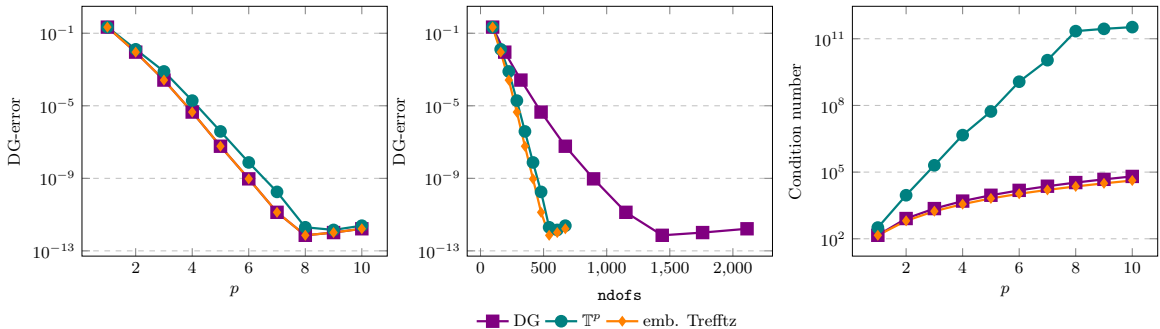
    
In \Cref{fig:lap} we observe exponential convergence in terms of $p$ and $\ndofs$.
In \cite{HMPS14} exponential convergence for harmonic polynomials in terms of $\ndofs$ is seen to be superior to standard polynomials, this can be seen in the center of \Cref{fig:lap}.
We measure the error in the DG-norm given by 
$$ \|w\|_{DG}=\sum_{K\in\Th}\|\nabla w\|_{L^2(K)}+\sum_{E\in\partial\Th} \|\sqrt{\alpha p^2h^{-1}}\jump{w}\|_{L^2(E)}.$$

In \Cref{fig:lap} we also compare the condition number of the different system matrices.
As expected from \Cref{lem:cond}, the conditioning of the embedded Trefftz method is bounded by the conditioning of the full system, and even outperforms it.
The very basic implementation of the harmonic polynomials that we used here appears very ill-conditioned.

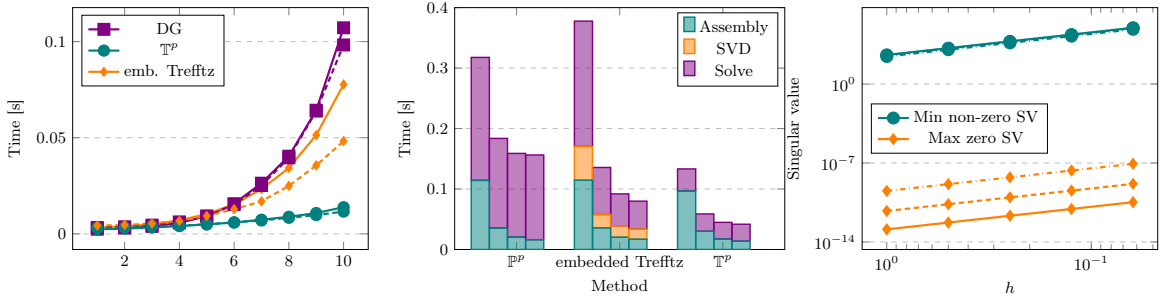
\begin{figure}[ht]
\begin{center}
\hspace*{-0.02\textwidth}\resizebox{1.025\linewidth}{!}{
\begin{tikzpicture}
    \begin{groupplot}[%
      group style={%
        group size=3 by 1,
        horizontal sep=1.5cm,
      },
    ymajorgrids=true,
    grid style=dashed,
    ]      

    \nextgroupplot[ylabel={Time [s]},yticklabels={0,0,0.05,0.1}, cycle list name=paulcolors, legend pos=north west]
    \foreach \h in {0.25}{
        \addplot+[discard if not={h}{\h}] table [x=p, y=time, col sep=comma] {results/lap2d4.csv};
        \addplot+[discard if not={h}{\h}] table [x=p, y=ttime, col sep=comma] {results/lap2d4.csv};
        \addplot+[discard if not={h}{\h}] table [x=p, y=svdttime, col sep=comma] {results/lap2d4.csv};
        \addplot+[discard if not={h}{\h}] table [x=p, y=time, col sep=comma] {results/lap2d8.csv};
        \addplot+[discard if not={h}{\h}] table [x=p, y=ttime, col sep=comma] {results/lap2d8.csv};
        \addplot+[discard if not={h}{\h}] table [x=p, y=svdttime, col sep=comma] {results/lap2d8.csv};
    }
    \legend{\DGmethod, \Tmethod, \ETmethod}

    \nextgroupplot[ybar stacked, 
    symbolic x coords={l2,svdt,trefftz},
    xticklabels={$\mathbb{P}^p$,embedded Trefftz,$\mathbb{T}^p$},
    xtick=data,
    xlabel={Method},
    ylabel={Time [s]},
    cycle list name=paulcolorsfill,
    enlarge x limits=.3, x=20mm,
    xshift=0.5*\pgfplotbarwidth,
    ymin=0,ymax=0.4,
    ]
    \addplot+[ybar,discard if not={p}{5},discard if not={hnr}{4}] table [x=method,y=assemble, col sep=comma] {results/lapbar2d8.csv};
    \addplot+[ybar,discard if not={p}{5},discard if not={hnr}{4}] table [x=method,y=svd, col sep=comma] {results/lapbar2d8.csv};
    \addplot+[ybar,discard if not={p}{5},discard if not={hnr}{4}] table [x=method,y=solve, col sep=comma] {results/lapbar2d8.csv};
    \legend{Assembly, SVD, Solve}

    \nextgroupplot[xmode=log,ymode=log, xlabel={$h$}, x dir=reverse, ylabel=Singular value, cycle list name=paulcolors2, 
        legend style={at={(0.03,0.5)},anchor=west}
    ]
    \foreach \p in {3,6,10}{
        \addplot+[discard if not={p}{\p}] table [x=h, y=minnz, col sep=comma] {results/lap2d8.csv};
        \addplot+[discard if not={p}{\p}] table [x=h, y=maxz, col sep=comma] {results/lap2d8.csv};
        \legend{Min non-zero SV, Max zero SV}
    }

\end{groupplot}

    \begin{groupplot}[%
      group style={%
        group size=2 by 1,
        horizontal sep=1.5cm,
      },
    ]      
    \nextgroupplot[group/empty plot,ymin=0,ymax=1,xmin=0,xmax=1]
    \nextgroupplot[ybar stacked, 
    symbolic x coords={l2,svdt,trefftz},
    xtick=data,
    cycle list name=paulcolorsfill,
    enlarge x limits=.3, x=20mm,
    xshift=1.5*\pgfplotbarwidth,
    ymin=0,ymax=0.4,hide axis
    ]
    \addplot+[ybar,discard if not={p}{5},discard if not={hnr}{4}] table [x=method,y=assemble, col sep=comma] {results/lapbar2d12.csv};
    \addplot+[ybar,discard if not={p}{5},discard if not={hnr}{4}] table [x=method,y=svd, col sep=comma] {results/lapbar2d12.csv};
    \addplot+[ybar,discard if not={p}{5},discard if not={hnr}{4}] table [x=method,y=solve, col sep=comma] {results/lapbar2d12.csv};
    \end{groupplot}

    \begin{groupplot}[%
      group style={%
        group size=2 by 1,
        horizontal sep=1.5cm,
      },
    ]      
    \nextgroupplot[group/empty plot,ymin=0,ymax=1,xmin=0,xmax=1]
    \nextgroupplot[ybar stacked, 
    symbolic x coords={l2,svdt,trefftz},
    xtick=data,
    cycle list name=paulcolorsfill,
    enlarge x limits=.3, x=20mm,
    xshift=-0.5*\pgfplotbarwidth,
    ymin=0,ymax=0.4,hide axis,
    ]
    \addplot+[ybar,discard if not={p}{5},discard if not={hnr}{4}] table [x=method,y=assemble, col sep=comma] {results/lapbar2d4.csv};
    \addplot+[ybar,discard if not={p}{5},discard if not={hnr}{4}] table [x=method,y=svd, col sep=comma] {results/lapbar2d4.csv};
    \addplot+[ybar,discard if not={p}{5},discard if not={hnr}{4}] table [x=method,y=solve, col sep=comma] {results/lapbar2d4.csv};
    \end{groupplot}

    \begin{groupplot}[%
      group style={%
        group size=2 by 1,
        horizontal sep=1.5cm,
      },
    ]      
    \nextgroupplot[group/empty plot,ymin=0,ymax=1,xmin=0,xmax=1]
    \nextgroupplot[ybar stacked, 
    symbolic x coords={l2,svdt,trefftz},
    xtick=data,
    cycle list name=paulcolorsfill,
    enlarge x limits=.3, x=20mm,
    xshift=-1.5*\pgfplotbarwidth,
    ymin=0,ymax=0.4,hide axis,
    ]
    \addplot+[ybar,discard if not={p}{5},discard if not={hnr}{4}] table [x=method,y=assemble, col sep=comma] {results/lapbar2d1.csv};
    \addplot+[ybar,discard if not={p}{5},discard if not={hnr}{4}] table [x=method,y=svd, col sep=comma] {results/lapbar2d1.csv};
    \addplot+[ybar,discard if not={p}{5},discard if not={hnr}{4}] table [x=method,y=solve, col sep=comma] {results/lapbar2d1.csv};
    \end{groupplot}
\end{tikzpicture}}
\end{center}    
\vspace*{-0.2cm}
\caption{Results for the Laplace problem in 2 dimensions. 
    Left: Comparison of the runtime on 4 threads (solid lines) and on 8 threads (dashed lines). 
    Center: Comparison of timings for the different steps of each method, for $p=5$ on a fixed mesh with $h=2^{-4}$. From left to right the bars correspond to 1,4,8, and 12 threads. 
    Right: Comparison of the singular values obtained when determining the null space of the operator for $p=3,6,10$ (full, dashed, dash-dotted line). 
}
\label{fig:lap2}
\end{figure}

We show plots of the runtime in \Cref{fig:lap2,fig:lap3} for the two and three dimensional example, respectively.
In \Cref{fig:lap2} it can be seen that the embedded Trefftz DG method benefits greatly from parallelization. 
The runtime is broken into parts of assembling the linear system and solving the linear system in \Cref{fig:lap2,fig:lap3}. 
For the embedded Trefftz DG method we plot also the time spent finding the local kernels of the operator using SVD, as described in \Cref{app:determinekernel}.
Note that in the case of the embedded Trefftz DG method the solving step includes the matrix multiplications needed in \eqref{eq:trefftzlinearsys}, therefore the time spent is not equal to that of the standard Trefftz method.
In 2d performing the SVD sequentially is time consuming, as shown in \Cref{fig:lap2}, a QR decomposition could improve the performance.
However, in 3d, see \Cref{fig:lap3}, the time spent on the SVD is completely negligible compared to the solver costs.

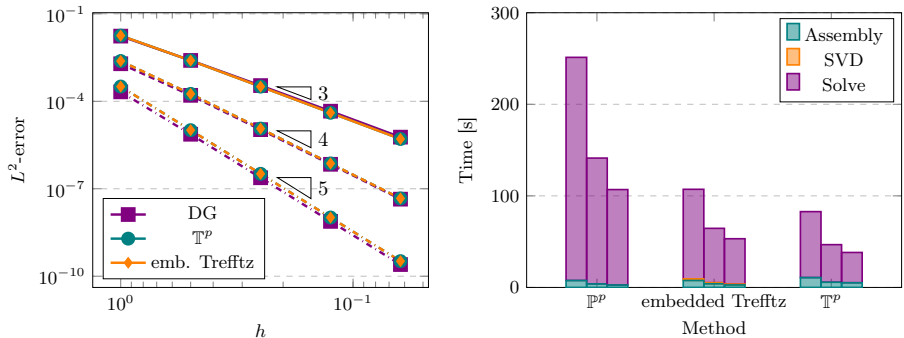
\begin{figure}[ht]
\begin{center}    
\resizebox{0.8\linewidth}{!}{
\begin{tikzpicture}
    \begin{groupplot}[%
      group style={%
        group size=2 by 1,
        horizontal sep=1.5cm,
      },
    ymajorgrids=true,
    grid style=dashed,
    ]      
    \nextgroupplot[xmode=log,ymode=log, xlabel={$h$}, x dir=reverse, ylabel=$L^2$-error, cycle list name=paulcolors, 
        legend pos=south west
    ]
    \foreach \p in {2,3,4}{
        \addplot+[discard if not={p}{\p}] table [x=h, y=error, col sep=comma] {results/lap3d8.csv};
        \addplot+[discard if not={p}{\p}] table [x=h, y=terror, col sep=comma] {results/lap3d8.csv};
        \addplot+[discard if not={p}{\p}] table [x=h, y=svdterror, col sep=comma] {results/lap3d8.csv};
        \legend{\DGmethod, \Tmethod, \ETmethod}
    }
    \logLogSlopeTriangle{0.65}{0.1}{0.73}{3}{black}; 
    \logLogSlopeTriangle{0.65}{0.1}{0.57}{4}{black};
    \logLogSlopeTriangle{0.65}{0.1}{0.4}{5}{black};
    \nextgroupplot[ybar stacked, 
    symbolic x coords={l2,svdt,trefftz},
    xticklabels={$\mathbb{P}^p$,embedded Trefftz,$\mathbb{T}^p$},
    xtick=data,
    xlabel={Method},
    ylabel={Time [s]},
    cycle list name=paulcolorsfill,
    enlarge x limits=.3, x=20mm,
    xshift=0.5*\pgfplotbarwidth,
    ymin=0,ymax=300,
    ]
    \addplot+[ybar,discard if not={hnr}{3},discard if not={p}{5}] table [x=method,y=assemble, col sep=comma] {results/lapbar3d8.csv};
    \addplot+[ybar,discard if not={hnr}{3},discard if not={p}{5}] table [x=method,y=svd, col sep=comma] {results/lapbar3d8.csv};
    \addplot+[ybar,discard if not={hnr}{3},discard if not={p}{5}] table [x=method,y=solve, col sep=comma] {results/lapbar3d8.csv};
    \legend{Assembly, SVD, Solve}
    \end{groupplot}

    \begin{groupplot}[%
      group style={%
        group size=2 by 1,
        horizontal sep=1.5cm,
      },
    ]      
    \nextgroupplot[group/empty plot,ymin=0,ymax=1,xmin=0,xmax=1]
    \nextgroupplot[ybar stacked, 
    symbolic x coords={l2,svdt,trefftz},
    xtick=data,
    cycle list name=paulcolorsfill,
    enlarge x limits=.3, x=20mm,
    xshift=1.5*\pgfplotbarwidth,
    ymin=0,ymax=300,hide axis
    ]
    \addplot+[ybar,discard if not={hnr}{3},discard if not={p}{5}] table [x=method,y=assemble, col sep=comma] {results/lapbar3d12.csv};
    \addplot+[ybar,discard if not={hnr}{3},discard if not={p}{5}] table [x=method,y=svd, col sep=comma] {results/lapbar3d12.csv};
    \addplot+[ybar,discard if not={hnr}{3},discard if not={p}{5}] table [x=method,y=solve, col sep=comma] {results/lapbar3d12.csv};
    \end{groupplot}

    \begin{groupplot}[%
      group style={%
        group size=2 by 1,
        horizontal sep=1.5cm,
      },
    ]      
    \nextgroupplot[group/empty plot,ymin=0,ymax=1,xmin=0,xmax=1]
    \nextgroupplot[ybar stacked, 
    symbolic x coords={l2,svdt,trefftz},
    xtick=data,
    cycle list name=paulcolorsfill,
    enlarge x limits=.3, x=20mm,
    xshift=-0.5*\pgfplotbarwidth,
    ymin=0,ymax=300,hide axis,
    ]
    \addplot+[ybar,discard if not={hnr}{3},discard if not={p}{5}] table [x=method,y=assemble, col sep=comma] {results/lapbar3d4.csv};
    \addplot+[ybar,discard if not={hnr}{3},discard if not={p}{5}] table [x=method,y=svd, col sep=comma] {results/lapbar3d4.csv};
    \addplot+[ybar,discard if not={hnr}{3},discard if not={p}{5}] table [x=method,y=solve, col sep=comma] {results/lapbar3d4.csv};
    \end{groupplot}
\end{tikzpicture}}
\end{center}    
\vspace*{-0.2cm}
\caption{Results for the Laplace problem in 3 dimensions with exact solution \eqref{eq:lap3d}. On the left: $h$-convergence for $p=2,3,4$. On the right: Comparison of timings for the different steps of each method, for $p=5$ on a fixed mesh with $h=2^{-3}$. The bars from left to right correspond to computations using $4,8,12$ threads for each method.}
\label{fig:lap3}
\end{figure}

It is possible to automatically find the dimension of the Trefftz space when computing the null space of the local operator.
To account for numerical errors in the computations of the singular values we choose a threshold of $\varepsilon = 10^{-7}$, identifying values below as zero singular values. 
In \Cref{fig:lap2} we plot the largest singular value that still needs to be identified as a zero singular value, as well as the smallest non-zero singular value.
Note that prior knowledge of the dimension of the Trefftz space, which we do not exploit here in the implementation, would eliminate this problem completely, cf. \cref{sec:dimofk}.

In \cite{LiShu2006,LiShu2012} convergence rates in $h$ for a mixed formulation over harmonic polynomials are shown.
The results in \Cref{fig:lap3} show that we recover the expected convergence rate of $\|u-u_h\|_{L^2(\dom)}=\calO(h^{\min(m,p)+1})$ for $u\in H^m(\dom)$.

\subsection{Poisson equation}\label{sec:poi}
Now, we consider the Poisson equation with Dirichlet boundary conditions
\begin{align*}
    \begin{cases}
    \Delta u = f &\text{ in } \dom, \\
    u=g &\text{ on } \partial \dom.
    \end{cases}
\end{align*}
We can use the symmetric IP-DG discretization given in \eqref{eq:dglap} with the new right hand side
\begin{align}\label{eq:poirhs}
    \ell(v) &= \int_\dom fv\ dV + \int_{\Fh^\text{bnd}}\left(\frac{\alpha p^2}{h} gv -\nx\cdot\nabla vg\right) dS.
\end{align}
For the numerical example we set the right hand side and the boundary conditions such that the exact solution is given by
\begin{align}
    & u=\sin(x)\sin(y)\sin(z) && \text{on }\dom=(0,1)^3.\label{poisol3}
\end{align}
To apply the embedded Trefftz method we use the approach described in \cref{sec:method3} to find a particular solution and homogenize the system.

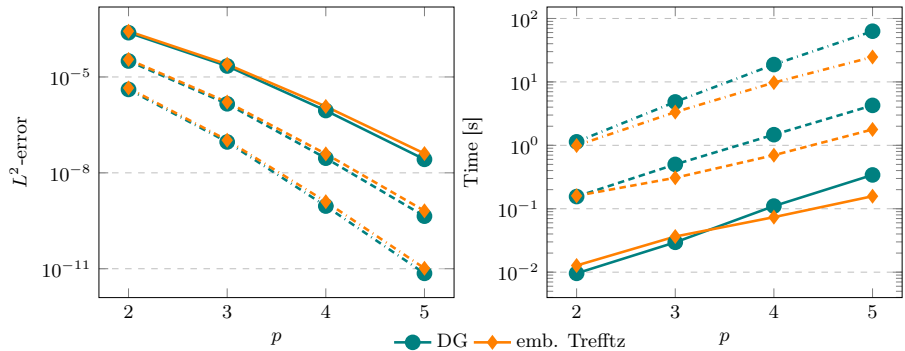
\begin{figure}[ht]
\resizebox{0.8\linewidth}{!}{
\begin{tikzpicture}
    \begin{groupplot}[%
      group style={%
        group name={my plots},
        group size=2 by 1,
        horizontal sep=1.5cm,
      },
    legend style={
        legend columns=2,
        at={(xticklabel cs:-0.1,0)},
        anchor=north,
        draw=none
    },
    xlabel={$h$},
    ymajorgrids=true,
    grid style=dashed,
    cycle list name=paulcolors2,
    ]      
    \nextgroupplot[ymode=log,xlabel={$p$}, ylabel={$L^2$-error}]
        \foreach \h in {0.5,0.25,0.125}{
            \addplot+[discard if not={h}{\h}] table [x=p, y=l2error, col sep=comma] {results/poi3d12.csv};
            \addplot+[discard if not={h}{\h}] table [x=p, y=svdt2error, col sep=comma] {results/poi3d12.csv};
        }
    \nextgroupplot[ymode=log, xlabel={$p$},ylabel={Time [s]}]
        \foreach \h in {0.5,0.25,0.125}{
            \addplot+[discard if not={h}{\h}] table [x=p, y=l2time, col sep=comma] {results/poi3d12.csv};
            \addplot+[discard if not={h}{\h}] table [x=p, y=svdt2time, col sep=comma] {results/poi3d12.csv};
        }
        \legend{\DGmethod,\ETmethod}
        \end{groupplot}
\end{tikzpicture}}
\caption{Numerical results for Poisson equation in 3 dimensions for meshes $h=0.5,\ 0.25,\ 0.125$ (solid, dashed, dotted dashed line, respectively). Left: $p$-convergence to the exact solution \eqref{poisol3}. Right: Timings on 12 threads.}
    \label{fig:poi}
\end{figure}

Results in terms of accuracy and computing time are shown in \Cref{fig:poi}. 
The embedded Trefftz DG method matches the convergence rates of the DG scheme using the standard polynomial space.
\rv{We consider the runtime for fixed mesh sizes and varying polynomial degree from $p=1,\dots,7$.
As discussed in \cref{sec:complexity} we expect worse performance for sparse meshes and large polyomial degree, due to the increased cost of computing the local kernel of the operator.
We can observe this on the mesh of the unit cube with mesh size $h=0.5$, where the embedded Trefftz method does not show significant improvement of the runtime.
The runtime on the fine mesh with $h=0.125$ is considerably improved even for rather large values of $p$.  
}

\subsection{\rv{Poisson equation with varying coefficient}}\label{sec:poivar}
We consider the Poisson equation with varying coefficients and Dirichlet boundary conditions, given by
\begin{align*}
    \begin{cases}
        \nabla\cdot(\bM\nabla u) = f &\text{ in } \dom, \\
    u=g &\text{ on } \partial \dom.
    \end{cases}
\end{align*}
We can use the symmetric IP-DG discretization given in \eqref{eq:dglap} and right hand side \eqref{eq:poirhs}, by inserting $\bM$ at the appropriate locations.
For the numerical example we fix the right hand side $f$ and the boundary conditions such that the exact solution is given by
\begin{align}\label{eq:poivarsol}
    & u=\sin(x)\sin(y) && \text{with} && \bM=\begin{pmatrix}1+x&0\\ 0&1+y \end{pmatrix}, && \text{and} && \dom=(0,1)^2.
\end{align}
To apply the embedded Trefftz method we apply again use the approach detailed in \cref{sec:method3} to homogenize the system.
Due to the varying coefficient we cannot find a standard Trefftz space, thus we follow the approach details in \Cref{sec:method2}, embedding a weak Trefftz space.
To construct the Trefftz embedding we apply \cref{eq:weakTproj,def:W2}, resulting in 
\begin{equation} \label{eq:Wpoivar}
    (\bW)_{ij} =\inner{-\nabla\cdot(\bM\nabla \phi_j),\psi_i},\quad \psi_i\in V^q(\Th).
\end{equation}
We compare different choices for the polynomial degree of the space $V^q$, with $q=p-1,\ p-2,\ p-3$. 
Note that the considerations in \Cref{sec:method2} imply that the optimal order is given by $q=p-2$ since $\Delta:V^{p}\rightarrow V^{p-2}$.

\begin{figure}[ht]
\resizebox{0.8\linewidth}{!}{
\begin{tikzpicture}
    \begin{groupplot}[%
      group style={%
        group name={my plots},
        group size=2 by 1,
        horizontal sep=1.5cm,
      },
    legend style={
        legend columns=4,
        at={(-0.1,-0.2)},
        anchor=north,
        draw=none
    },
    xlabel={$h$},
    ymajorgrids=true,
    grid style=dashed,
    cycle list name=paulcolors4,
    ]      
    \nextgroupplot[xmode=log,ymode=log,xlabel={$h$},x dir=reverse, ylabel={$L^2$-error}]
        \foreach \p in {7}{
            \addplot+[discard if not={p}{\p}] table [x=h, y=error, col sep=comma] {results/poivar2d12.csv};
            \addplot+[discard if not={p}{\p}] table [x=h, y=svdterror1, col sep=comma] {results/poivar2d12.csv};
            \addplot+[discard if not={p}{\p}] table [x=h, y=svdterror2, col sep=comma] {results/poivar2d12.csv};
            \addplot+[discard if not={p}{\p}] table [x=h, y=svdterror3, col sep=comma] {results/poivar2d12.csv};
        }
    \logLogSlopeTriangle{0.65}{0.1}{0.4}{8}{black};
    \nextgroupplot[xlabel={$p$},ylabel={\ndofs}]
        \foreach \h in {0.125}{
            \addplot+[discard if not={h}{\h}] table [x=p, y=ndof, col sep=comma] {results/poivar2d12.csv};
            \addplot+[discard if not={h}{\h}] table [x=p, y=tndof1, col sep=comma] {results/poivar2d12.csv};
            \addplot+[discard if not={h}{\h}] table [x=p, y=tndof2, col sep=comma] {results/poivar2d12.csv};
            \addplot+[discard if not={h}{\h}] table [x=p, y=tndof3, col sep=comma] {results/poivar2d12.csv};
        }
        \legend{\DGmethod, \ETmethod\ $q=p-1$, \ETmethod\ $q=p-2$, \ETmethod\ $q=p-3$}
        \end{groupplot}
\end{tikzpicture}}
\caption{Numerical results for the Poisson equation with varying coefficient comparing different choices of $q$ in \eqref{eq:Wpoivar}. Results are obtained in 2 dimensions for meshes $h=1,0.5,\ 0.25,\ 0.125$. 
Left: $h$-convergence for $p=6$ to the exact solution \eqref{eq:poivarsol}. 
Right: Comparison of the global number of degrees of freedom (\ndofs) for different polynomial degrees on a fixed mesh with $h=0.125$.
}
    \label{fig:poivar}
\end{figure}
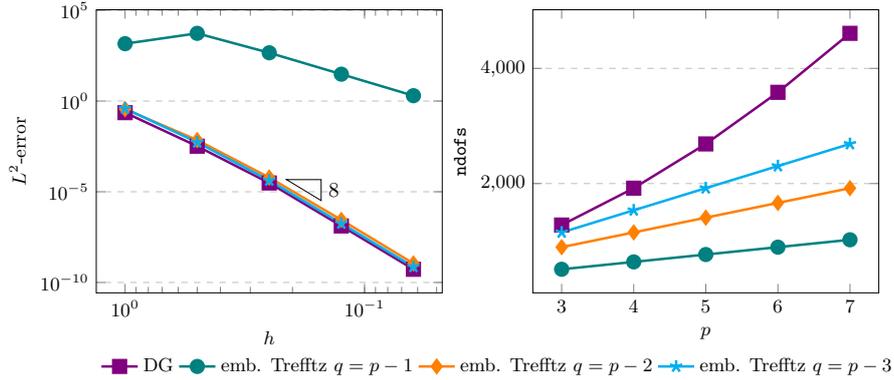

The results presented in \Cref{fig:poivar} show the expected behavior.
Testing with an increased number of test functions, using $q=p-1$, leads to an embedding of a smaller subspace than the actual weak Trefftz space. 
This results in a loss of the good approximation properties, as we see in \Cref{fig:poivar} on the left.
Using a smaller test space $V^q$ with $q=p-3$, is less efficient since the approximation space holds more degrees of freedom, as shown in \Cref{fig:poivar} on the right. 
While the error for the choices $q=p-2$ and $q=p-3$ shows slight differences, most likely due to different conditioning, the behavior is the same, as predicted by \cref{lem:cea}.
Specifically, the different computation of the particular solution does not affect the approximation. 
Recall that the particular solutio satisfies $\Pi \calL u_{h,f} = \Pi f$, i.e. $\inner{\calL u_{h,f},v}=\inner{f,v}\, \forall v\in V^\bq$, as described in \cref{sec:method3}.
Hence, for the choice of $q=p-3$, the approximation of the particular solution might be worse than for the choice $q=p-2$, however, the loss in accuracy is made up for by the larger Trefftz space used after homogenization.

\subsection{Acoustic wave equation with piecewise constant wave speed}\label{sec:wave}
Next, we consider the first order wave equation in a space-time setting, given by
\begin{align}\label{eq:firstordwaveeq}
    \begin{cases}
        \nabla \cdot \bm\sigma + c^{-2} \frac{\partial v}{\partial t} = 0  & \qquad \text{ in } \dom\times [0,T],\\
        \nabla v + \frac{\partial \bm\sigma}{\partial t} = \bm 0 & \qquad \text{ in } \dom\times [0,T], \\
        v(\cdot,0) = v_0,\ \bm\sigma(\cdot,0) = \bm\sigma_0 & \qquad \text{ on } \Omega\times\{0\}, \\
        v=g_D & \qquad \text{ on } \partial\dom \times [0,T],
    \end{cases}
\end{align}
where
$(\bm\sigma,v)$ are the unknowns - acoustic speed and pressure,
$(\bm\sigma_0,v_0)$ are the initial conditions,
$g_D$ is Dirichlet boundary data,
and $T$ is the final time. 
In this section we will assume that the wavespeed $c$ is piecewise constant.

To apply our framework we take
\begin{align*}
    \calL=
\begin{pmatrix}
    \nabla\cdot & c^{-2}\dt \\
    \nabla & \dt
\end{pmatrix}
\qquad
u=
\begin{pmatrix}
    \bm\sigma\\
    v
\end{pmatrix}
\end{align*}
The local space-time Trefftz space was introduced and analyzed in \cite{mope18}, and is given by
\begin{align}\label{eq:Twave}
\begin{aligned}
\IT^p(K)=&
\left\{\wt\in\IP^p(K)^{n+1} \Big|
\begin{array}{l}
\nabla w+\partial_t\bm\tau = \bm 0 \\
\nabla\cdot\bm\tau+ c^{-2}\partial_t w = 0 
\end{array}
\right\}
\end{aligned}
\end{align}
We consider the space-time DG-scheme used in \cite{StockerSchoeberl,BMPS20,mope18,qtrefftz}:
\begin{align}\begin{split}
    & \text{Find } (v_{hp},\bm\sigma_{hp})\in (\Vhp)^{d+1} \quad\text{s.t.} \\
    & a_h(v_{hp},\bm\sigma_{hp};w,\bm\tau) = \ell(w,\bm\tau) \qquad \forall (w,\bm\tau)\in(\Vhp)^{d+1},
\end{split}\end{align}
with
\begin{align}  
    \label{eq:dgwave}
    a_h(v_{hp},\bm\sigma_{hp};w,\bm\tau ) &= -\sum_{K\in\Fh}\int_K \left(v_{hp}\left(\nabla\cdot\bm\tau+c^{-2}\partial_t w\right)+\bm\sigma_{hp}\cdot\left(\partial_t\bm\tau+\nabla w\right)\right) dV\\
                                                  &\quad+\int_{\mathcal{F}_h^{\text{space}}} \left( c^{-2} v_{hp}^- \jump{ w}_t +\bm\sigma^-_{hp}\cdot \jump{\bm\tau}_t + v^-_{hp} \jump{\bm\tau}_N+\bm\sigma^-_{hp} \cdot \jump{ w}_N \right)\ dS \nonumber\\
                                                  &\quad + \int_{\mathcal{F}_h^{\text{time}}} \left( \avg{ v_{hp}} \jump{ \bm\tau}_N + \avg{\bm\sigma_{hp}}\cdot\jump{ w}_N + \alpha \jump{ v_{hp}}_N\cdot\jump{ w }_N + \beta \jump{\bm\sigma_{hp}}_N\jump{\bm\tau}_N \right)\ dS \nonumber\\
    &\quad + \int_{\mathcal{F}^T_h} c^{-2}v_{hp} w + \bm\sigma_{hp}\cdot\bm\tau\ dS + \int_{\mathcal{F}^D_h} (\bm\sigma\cdot\bm n^x_\Omega+\alpha v_{hp} ) w \ dS\nonumber\\
\end{align}  
and
\begin{align*}  
    \ell(w,\bm \tau) &= \int_{\mathcal{F}_h^0} c^{-2} v_0 w + \bm\sigma\cdot\bm\tau \ dS + \int_{\mathcal{F}_h^D} g_D(\alpha w - \bm\tau\cdot\bm n^x_\Omega)\ dS.
\end{align*}


For the numerical example we set boundary and initial conditions such that the exact solution with wavespeed $c=1$ is given by
\begin{align}\label{eq:waveex}
    & v=\sqrt{2}\cos(\sqrt{2}t+x+y),\ \bm\sigma=(-\cos(\sqrt{2}t+x+y),-\cos(\sqrt{2}t+x+y)),\ \dom=(0,1)^2.
\end{align}
The penalization parameters in \eqref{eq:dgwave} are chosen as $\alpha=\beta=0.5$, as in \cite{StockerSchoeberl}.
The considered space--time meshes are made up from a simplicial unstructured mesh of the spatial domain and a tensor product mesh in time.
The height of the time slabs is chosen approximately as the mesh size of the spatial domain. 
In \cref{fig:wave} we observe that both Trefftz method maintain the accuracy of the underlying DG method while yielding almost an order of magnitude speed up in the computation time on the finest considered level. Also on the finest level the solver costs seem to dominate so that the embedded Trefftz DG method and the Trefftz method perform equally well.

\begin{figure}[ht]
    \begin{center}
\resizebox{0.8\linewidth}{!}{
\begin{tikzpicture}
    \begin{groupplot}[%
      group style={%
        group size=2 by 1,
        horizontal sep=1.5cm,
      },
    legend style={
        legend columns=3,
        at={(-0.2,-0.1)},
        anchor=north,
        draw=none
    },
    ymajorgrids=true,
    grid style=dashed,
    ]      
    \nextgroupplot[xmode=log,ymode=log, xlabel={$h$}, x dir=reverse, ylabel=$L^2$-error, cycle list name=paulcolors]
    \foreach \p in {2,3,4}{
        \addplot+[discard if not={p}{\p}] table [x=h, y=error, col sep=comma] {results/wave2d24.csv};
        \addplot+[discard if not={p}{\p}] table [x=h, y=terror, col sep=comma] {results/wave2d24.csv};
        \addplot+[discard if not={p}{\p}] table [x=h, y=svdterror, col sep=comma] {results/wave2d24.csv};
    }
    \logLogSlopeTriangle{0.66}{0.1}{0.72}{3}{black}; 
    \logLogSlopeTriangle{0.66}{0.1}{0.58}{4}{black};
    \logLogSlopeTriangle{0.66}{0.1}{0.43}{5}{black};
    \nextgroupplot[xmode=log,ymode=log, xlabel={$h$}, x dir=reverse, ylabel={Time [s]}, cycle list name=paulcolors4]
    \foreach \p in {4}{
        \addplot+[discard if not={p}{\p}] table [x=h, y=time, col sep=comma] {results/wave2d24.csv};
        \addplot+[discard if not={p}{\p}] table [x=h, y=ttime, col sep=comma] {results/wave2d24.csv};
        \addplot+[discard if not={p}{\p}] table [x=h, y=svdttime, col sep=comma] {results/wave2d24.csv};
    }
    \legend{\DGmethod, \Tmethod, \ETmethod}
    \end{groupplot}
\end{tikzpicture}}
\end{center}
\caption{Numerical results for the wave equation in 2+1 dimensions. On the left: $h$-convergence in the space-time $L^2$-error for $p=2,3,4$ corresponding to full, dashed and dash-dotted line. Convergence is given with respect to exact solution \eqref{eq:waveex}. On the right: timings for $p=4$ \rv{on 24 threads}.
}
\label{fig:wave}
\end{figure}
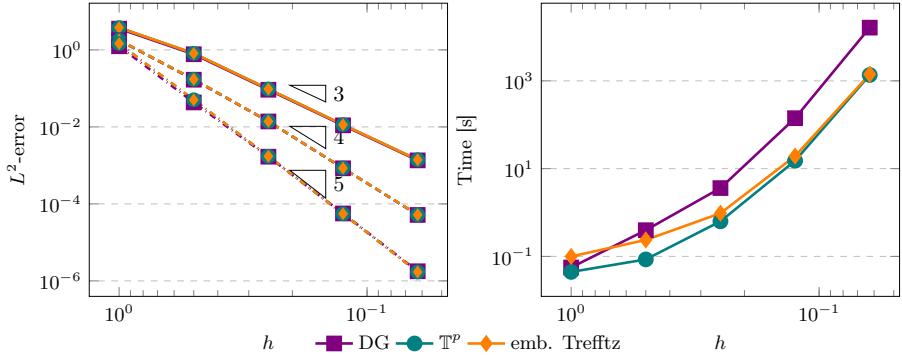

\subsection{Acoustic wave equation in inhomogeneous media}\label{sec:qwave}
Instead of the piecewise constant case from the previous section, we now consider \eqref{eq:firstordwaveeq} with wavespeed smoothly varying in space $c=c(\bm x)$.
The construction of a basis for \eqref{eq:Twave} with smooth wavespeed is non-viable. 
A space with similar properties, for which a basis can be constructed, has been introduced in \cite{qtrefftz} - the quasi-Trefftz space for an element $K\in\Th$ is given by
\begin{align}\label{eq:QT}
\begin{aligned}
\QT^p(K)\!\!:=\!&
\left\{\wt\!\in\IP^p(K)^{n+1} \Big|\!\!\!
\begin{array}{l}
D^\mi (\nabla w+\partial_t\bm\tau)(\bx_K,t_K) = \bm 0 \\
D^\mi (\nabla\cdot\bm\tau+ c(\bx)^{-2}\partial_t w)(\bx_K,t_K) = 0 
\end{array}
\!\forall \mi\!\in\!\IN^{n+1}_0\!, |\mi|\le p\!-\!1\!\right\}\!.\!\!
\end{aligned}
\end{align}
\rv{
where we use the notation $\mi=(\mi_\bx,i_t)=(i_{x_1},\ldots,i_{x_n},i_t)
\in \IN^{n+1}_0$ for integer non-negative multi-indices and 
$D^\mi f := 
\partial_{x_1}^{i_{x_1}}\cdots\partial_{x_n}^{i_{x_n}}\partial_{t}^{i_t} f$ for the derivatives.
The polynomials in the space are constructed such that the Taylor polynomial of their image under the wave operator vanishes at the element center $(\bx_K,t_K)$ up to order $p-2$.

The conditions of the quasi-Trefftz space are formulated in a way to allow for recursive construction of the basis functions. It is possible to write this space in the way of \eqref{eq:weakTspace}, see \cref{rem:qtrefftz}. However, this is not the most practical way to implement the embedding. Thus, we proceed with the construction proposed in \cref{eq:weakTproj,def:W2}.}
While an explicit construction of the weak Trefftz space is unfeasible, the embedded Trefftz method is still able to project on such a space. 
We construct the weak Trefftz embedding as described in \cref{sec:method2,sec:algo}, using the following condition for the weak Trefftz space
$\inner{\calL (\bm\sigma,v),(\bm\tau,w)}=0,\ \forall (\bm\tau,w)\in (V^{p-1}(\Th))^{d+1}$.
The volume term in $\eqref{eq:dgwave}$ can be dropped when using the embedded Trefftz DG method, which we also did in the numerical examples in this section, cf. \cref{rem:qtrefftzterms}.

We consider the exact solution given by
\begin{align}\label{eq:qwaveex}
    &v=-\sqrt{2\kappa(\kappa-1)}(x+y+1)^\kappa \ee^{-\sqrt{2\kappa(\kappa-1)}t}, \quad
    \bm\sigma=\begin{pmatrix} -\kappa(x+y+1)^{\kappa-1}  \ee^{-\sqrt{2\kappa(\kappa-1)}t}\\ -\kappa(x+y+1)^{\kappa-1}  \ee^{-\sqrt{2\kappa(\kappa-1)}t} \end{pmatrix}
\end{align}
with $\kappa=2.5$ and wavespeed $c(x,y)=x+y+1$ on the space--time domain $\dom\times(0,1)$ with $\dom=(0,1)^2$.

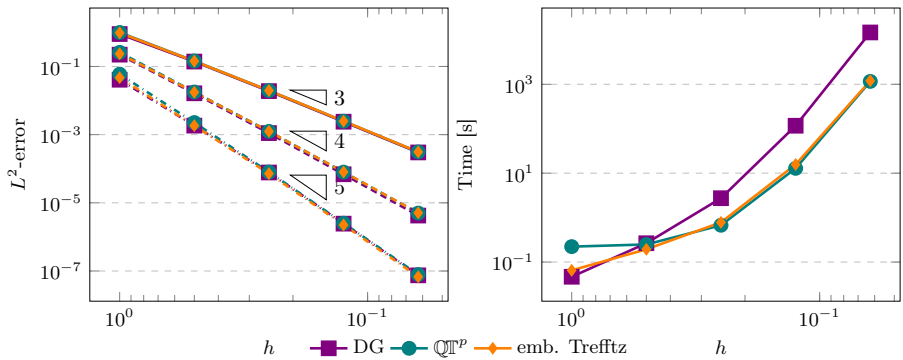
\begin{figure}[ht]
\resizebox{0.8\linewidth}{!}{
\begin{tikzpicture}
    \begin{groupplot}[%
      group style={%
        group size=2 by 1,
        horizontal sep=1.5cm,
      },
    legend style={
        legend columns=3,
        at={(-0.2,-0.1)},
        anchor=north,
        draw=none
    },
    ymajorgrids=true,
    grid style=dashed,
    ]      
    \nextgroupplot[xmode=log,ymode=log, xlabel={$h$}, x dir=reverse, ylabel=$L^2$-error, cycle list name=paulcolors]
    \foreach \p in {2,3,4}{
        \addplot+[discard if not={p}{\p}] table [x=h, y=error, col sep=comma] {results/qwave2d24.csv};
        \addplot+[discard if not={p}{\p}] table [x=h, y=terror, col sep=comma] {results/qwave2d24.csv};
        \addplot+[discard if not={p}{\p}] table [x=h, y=svdterror, col sep=comma] {results/qwave2d24.csv};
    }
    \logLogSlopeTriangle{0.66}{0.1}{0.72}{3}{black}; 
    \logLogSlopeTriangle{0.66}{0.1}{0.58}{4}{black};
    \logLogSlopeTriangle{0.66}{0.1}{0.43}{5}{black};
    \nextgroupplot[xmode=log,ymode=log, xlabel={$h$}, x dir=reverse, ylabel={Time [s]}, cycle list name=paulcolors,]
    \foreach \p in {4}{
        \addplot+[discard if not={p}{\p}] table [x=h, y=time, col sep=comma] {results/qwave2d24.csv};
        \addplot+[discard if not={p}{\p}] table [x=h, y=ttime, col sep=comma] {results/qwave2d24.csv};
        \addplot+[discard if not={p}{\p}] table [x=h, y=svdttime, col sep=comma] {results/qwave2d24.csv};
    }
    \legend{\DGmethod, \QTmethod, \ETmethod}
    \end{groupplot}
\end{tikzpicture}}
\caption{Numerical results for the wave equation with smooth coefficient in 2+1 dimensions, the exact solution given by \eqref{eq:qwaveex}. On the left: $h$-convergence for $p=2,3,4$ corresponding to full, dashed and dash-dotted line. On the right: timings for $p=4$ \rv{on 24 threads}.}
\label{fig:qwave}
\end{figure}
In \cref{fig:qwave} we observe a similar performance as for the homogeneous case.
When looking carefully we can now see a difference between Quasi-Trefftz and embedded Trefftz DG method.
While in the homogeneous case the embedded Trefftz space coincides with the Trefftz space, in the inhomogeneous case the two spaces do not coincide, compare \eqref{eq:weakTspace} and \eqref{eq:QT}.
However, the asymptotics seem to be unaffected and all methods perform equally well in terms of accuracy on a given mesh. 
The timings are again in agreement with the experience from the homogeneous case.

\begin{remark}\label{rem:qtrefftzterms}
    The volume term in \eqref{eq:dgwave} can be omitted for the case of homogeneous media when using Trefftz or embedded Trefftz DG methods. 
    In the inhomogeneous case, in \cite{qtrefftz} the quasi-Trefftz methods also require the volume terms for stability, however, even in the case of smooth coefficients, these terms can still be omitted when using embedded Trefftz DG method as described in \cref{sec:qwave}.
    Note that in \cite{qtrefftz,BMPS20} there appear additional Galerkin-least squares terms needed for the analysis, however, already in \cite{qtrefftz} it was observed that they do not seem to play a significant role in the numerical examples, which is why they are neglected here.
    \end{remark}
    
\subsection{Helmholtz equation}\label{sec:helm}
We now switch to the time-harmonic case of wave equations and consider the Helmholtz equation with Robin boundary conditions
\begin{align*}
    \begin{cases}
    -\Delta u - \omega^2 u= 0 &\text{ in } \dom, \\
    \frac{\partial u}{\partial \nx} + i u = g &\text{ on } \partial \dom.
    \end{cases}
\end{align*}
We consider the DG-scheme used in \cite{cessenat1998application,ghp09,mps13,hmp11b,AndreaPhD} with (bi)linear forms
\begin{subequations}
\begin{align}
        a_h(u,v) &= \sum_{K\in\Th}\int_K \nabla u\nabla v-\omega^2 uv\ dV
        -\int_{\Fh^\text{int}}\left(\avg{\nabla u}\jump{v}+\jump{u} \avg{\overline{\nabla v}} \right) dS \nonumber \\ 
                 &\qquad+\int_{\Fh^\text{int}} \left( i\alpha \omega\jump{u}\jump{\overline{v}} - \frac{\beta}{i\omega}\jump{\nabla u}\jump{\overline{\nabla v}} \right) dS -\int_{\Fh^\text{bnd}}\delta\left(\nx\cdot\nabla u \overline{v}+u \overline{\nx\cdot\nabla v}\right) dS\\ \nonumber
                 &\qquad+\int_{\Fh^\text{bnd}} \left( i(1-\delta)\omega{u}{\overline{v}} - \frac{\delta}{i\omega}{\nabla u}{\overline{\nabla v}} \right) dS \\ 
        \ell(v) &= \int_{\Fh^\text{bnd}}\left( (1-\delta)g\overline{v} - \frac{\delta}{i\omega}g\overline{\nx\cdot\nabla v}\right) dS
\end{align}
\end{subequations}
with the choices $\alpha=\beta=\delta=0.5$ used in \cite{cessenat1998application}.
A Trefftz space for the Helmholtz equation in two dimensions is given by the (non-polynomial) space of plane wave functions
\begin{align}\label{eq:pwspace}
    \IT^p=\{e^{-i\omega(d_j\cdot \bm x)} \sst j=-p,\dots,p\}.
\end{align}
The combination of DG (bi)linear forms and this trial and test space is known as plane wave DG method. 
In the numerical experiment we consider the exact solution given by 
\begin{align}\label{eq:helmex}
    & u=H_0^{(1)}(\omega|\bm x-\bm x_0|),\quad \bm x_0=(-0.25,0), && \dom=(0,1)^2.
\end{align}
where $H_0^{(1)}$ is the zero-th order Hankel function of the first kind.
Results are shown in \cref{fig:helm} \rv{for $\omega=1$.}
Surprisingly, the embedded Trefftz DG solution is extremely close to the plane wave DG method although the one space is piecewise polynomial and the other one is not. 
While all methods exhibit the same convergence rate, both Trefftz methods yield the significantly smaller error compared to the DG method. 
This can certainly not be attributed to different approximation spaces as the embedded Trefftz DG space is a subspace of the DG space. Hence, the results suggest that the embedded Trefftz method has better stability properties than the standard DG method, probably comparable to that of the plane wave DG method, cf. also \cref{rem:tdgpwdg}.

In \cref{fig:helm} on the right, we consider $p$-convergence on two different meshes. 
On the finer mesh, with $h=2^{-3}$, the plane waves fail to converge for $p=5$. 
We have implemented the simplest form of the plane wave basis functions, given by the functions in \eqref{eq:pwspace}, which are notoriously haunted by ill-conditioning. 
While more stable constructions for the plane wave basis exist, see for example \cite[Sec. 3.4.1]{AndreaPhD}, it is interesting to note that the embedded Trefftz space shows improved conditioning without any additional effort.

\begin{figure}[ht]
\begin{center}    
\resizebox{0.8\linewidth}{!}{
    \begin{tikzpicture}[spy using outlines={circle, magnification=4, size=2cm, connect spies}]
    \begin{groupplot}[%
      group style={%
        group name={my plots},
        group size=2 by 1,
        horizontal sep=1.5cm,
      },
    legend style={
        legend columns=4,
        at={(-0.1,-0.1)},
        anchor=north,
        draw=none
    },
    xlabel={$h$},
    ymajorgrids=true,
    grid style=dashed,
    cycle list name=paulcolors,
    ]      
    \nextgroupplot[ymode=log,xmode=log,x dir=reverse, ylabel={$L^2$-error}]
        \foreach \p in {3,4}{
            \addplot+[discard if not={p}{\p}] table [x=h, y=error, col sep=comma] {results/helmholtz2d1.csv};
            \addplot+[discard if not={p}{\p}] table [x=h, y=terror, col sep=comma] {results/helmholtz2d1.csv};
            \addplot+[discard if not={p}{\p}] table [x=h, y=svdterror, col sep=comma] {results/helmholtz2d1.csv};
        }
        \logLogSlopeTriangle{0.8}{0.1}{0.6}{4}{black}; 
        \logLogSlopeTriangle{0.8}{0.1}{0.2}{5}{black};
    \nextgroupplot[ymode=log, xlabel={$p$},ylabel={$L^2$-error}]
        \foreach \h in {0.25,0.125}{
            \addplot+[discard if not={h}{\h}] table [x=p, y=error, col sep=comma] {results/helmholtz2d1.csv};
            \addplot+[discard if not={h}{\h}] table [x=p, y=terror, col sep=comma] {results/helmholtz2d1.csv};
            \addplot+[discard if not={h}{\h}] table [x=p, y=svdterror, col sep=comma] {results/helmholtz2d1.csv};
        };
        \legend{\DGmethod, \Tmethod, \ETmethod}
        \end{groupplot}
        \spy[purple,size=2cm] on (4.07,1.6) in node [fill=white] at (1.4,1.2);            
        \spy[purple,size=2cm] on (12.47,0.43) in node [fill=white] at (8.6,1.2);            
\end{tikzpicture}}
\end{center} \vspace*{-0.25cm}
\caption{Numerical results for Helmholtz equation for the exact solution given in \eqref{eq:helmex}. On the left: $h$-convergence for $p=2,3$ corresponding to the full and dashed line. On the right: $p$-convergence for $h=2^{-2},2^{-3}$ corresponding to the full and dashed line, respectively.}
    \label{fig:helm}
\end{figure}
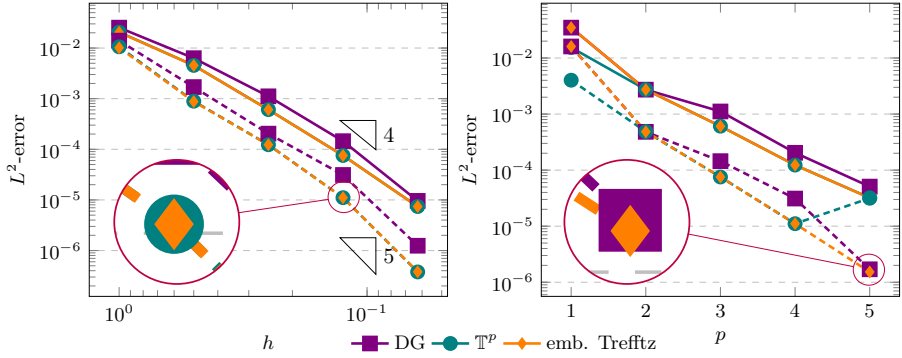

\begin{remark}[Approximation of plane wave functions with the embedded Trefftz DG space] \label{rem:tdgpwdg}
The remarkable proximity of both Trefftz DG method in the previous experiment may suggest that the embedded Trefftz space approximates the plane wave DG space.
We investigate this in more detail for a one-dimensional example. On $[0,1]$ we consider the embedded Trefftz space for $p=1,...,5$. The embedded Trefftz space is only two-dimensional as in 1D $\dim V^p - \dim V^q = 2$ (independent of $p$). On this space we approximate the two plane wave functions $\sin(\omega x)$ and $\cos(\omega x)$ for $\omega=2\pi$. The result is shown in \cref{fig:helmshapes} (top row). 
For comparison we approximate $\sin(\omega x)$ and $\cos(\omega x)$ on the (much) larger space $\Vhp$ (bottom row).
We observe that 
starting from $p=2$ the approximation quality of the embedded Trefftz space is close to the that of the DG space and 
for $p = 5$ the plane waves are extremely well resolved rendering the embedded Trefftz space in close proximity to the plane wave space. 
\end{remark}
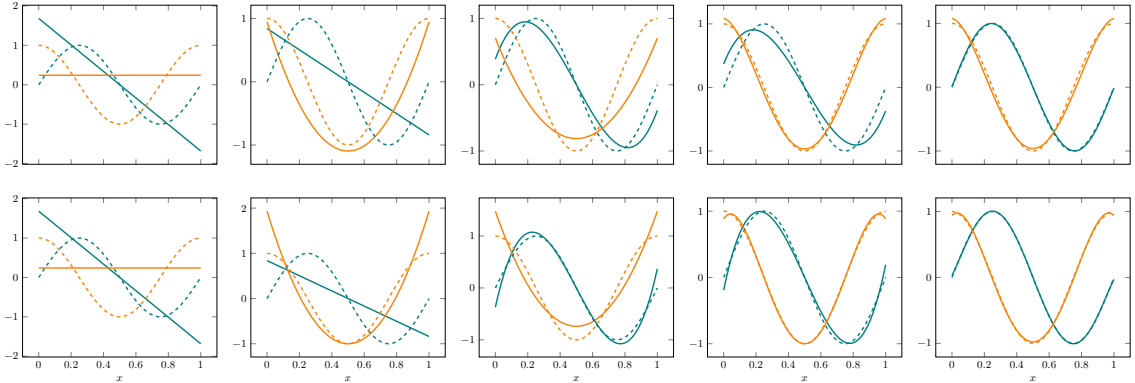
\begin{figure}[ht]
\resizebox{\linewidth}{!}{
\begin{tikzpicture}
    \begin{groupplot}[%
      group style={%
        group name=my plots,
        group size=5 by 2,
        xlabels at=edge bottom,
        x descriptions at=edge bottom,
      },
    xlabel={$x$},
    cycle list name=paulcolors2,
    ]      
    \def\myPlots{}
        \pgfplotsforeachungrouped \p in {1,2,3,4,5}{
        \eappto\myPlots{%
            \noexpand\nextgroupplot
            \noexpand\addplot+[] table [mark=none, x=px, y=p\p_tgfu0, col sep=comma] {results/helmshapes.csv};
            \noexpand\addplot+[] table [mark=none, x=px, y=p\p_tgfu1, col sep=comma] {results/helmshapes.csv};
            \noexpand\addplot+[dashed] table [mark=none, x=px, y=p\p_u0, col sep=comma] {results/helmshapes.csv};
            \noexpand\addplot+[dashed] table [mark=none, x=px, y=p\p_u1, col sep=comma] {results/helmshapes.csv};
        } 
        }
        \pgfplotsforeachungrouped \p in {1,2,3,4,5}{
        \eappto\myPlots{
            \noexpand\nextgroupplot
            \noexpand\addplot+[] table [mark=none, x=px, y=p\p_gfu0, col sep=comma] {results/helmshapes.csv};
            \noexpand\addplot+[] table [mark=none, x=px, y=p\p_gfu1, col sep=comma] {results/helmshapes.csv};
            \noexpand\addplot+[dashed] table [mark=none, x=px, y=p\p_u0, col sep=comma] {results/helmshapes.csv};
            \noexpand\addplot+[dashed] table [mark=none, x=px, y=p\p_u1, col sep=comma] {results/helmshapes.csv};
        } 
        }
        \myPlots
    \end{groupplot}
\end{tikzpicture}}
\vspace*{-0.5cm}
\caption{Approximation of the real part of plane wave basis functions (dashed line) in 1 dimension by the embedded Trefftz basis 
in the top row and by full polynomial space in the bottom row (solid line)
, for $p=1,\dots,5$.}
    \label{fig:helmshapes}
\end{figure}

\subsection{Linear transport equation}\label{sec:adv}
In this section we finally consider an example that is typically not related to Trefftz method: A scalar linear transport problem, the advection equation. It reads as
\begin{align*}
   \bb\cdot\nabla  u &= f\quad \text{ in } \Omega,\\  
    u &= u_D \quad \text{ on }  \partial \Omega_{\text{in}} :=\{\bx\in\partial \Omega\mid \bb\cdot\nx < 0\}.
\end{align*} 
for a given velocity field $\bb$ which we assume to be divergence-free. 
As underlying DG discretization we choose the standard Upwind DG formulation which reads as
\begin{subequations}
    \label{eq:dgupw}    
\begin{align}
    a_h(u,v) &= \sum_{K\in\Th} \Big\{  \int_K  - u ~ \bb\cdot \nabla v\ dV +\int_{\partial K\setminus {\partial \Omega}{\text{in}} } \bb\cdot\nx\hat{u} v~  dS   \Big\}  \\ 
    \ell(v) &= \sum_{K\in\Th} \int_K  f v\ dV -
    \int_{{\partial \Omega}{\text{in}}} \bb \cdot \nx u_D v~ dS
\end{align}
\end{subequations} 
Here, we used the upwind notation $\hat{u}(\bx) = \lim_{t\to 0^+}  u(\bx - \bb t )$.   
We do not assume $\bb$ to be piecewise constant so that there will not be a suitable polynomial Trefftz space in general.  
To $\calL =\bb\cdot\nabla $ we hence choose the weak Trefftz space with $V^q$ with $q=p-1$.
The local Trefftz space on an element $K$ will hence have the dimension $ M = \#\mathcal{P}^p(K) - \#\mathcal{P}^{p-1}(K)$ which in 1D is $M=1$, in 2D is $M=p+1$ and in 3D is $(p+1)(p+2)/2$. In comparison to scalar second order problems we hence only have approximately half the degrees of freedoms in the resulting weak Trefftz space. In \cref{fig:advshapes} we display a set of weak Trefftz basis functions that is obtained in 2D for a non-constant flow field and $p=4$. The basis functions are approximately - but not exactly - constant along the flow trajectories.
\begin{figure}
    \begin{center}
        \hspace*{-0.02\textwidth}
        \includegraphics[width=0.215\textwidth]{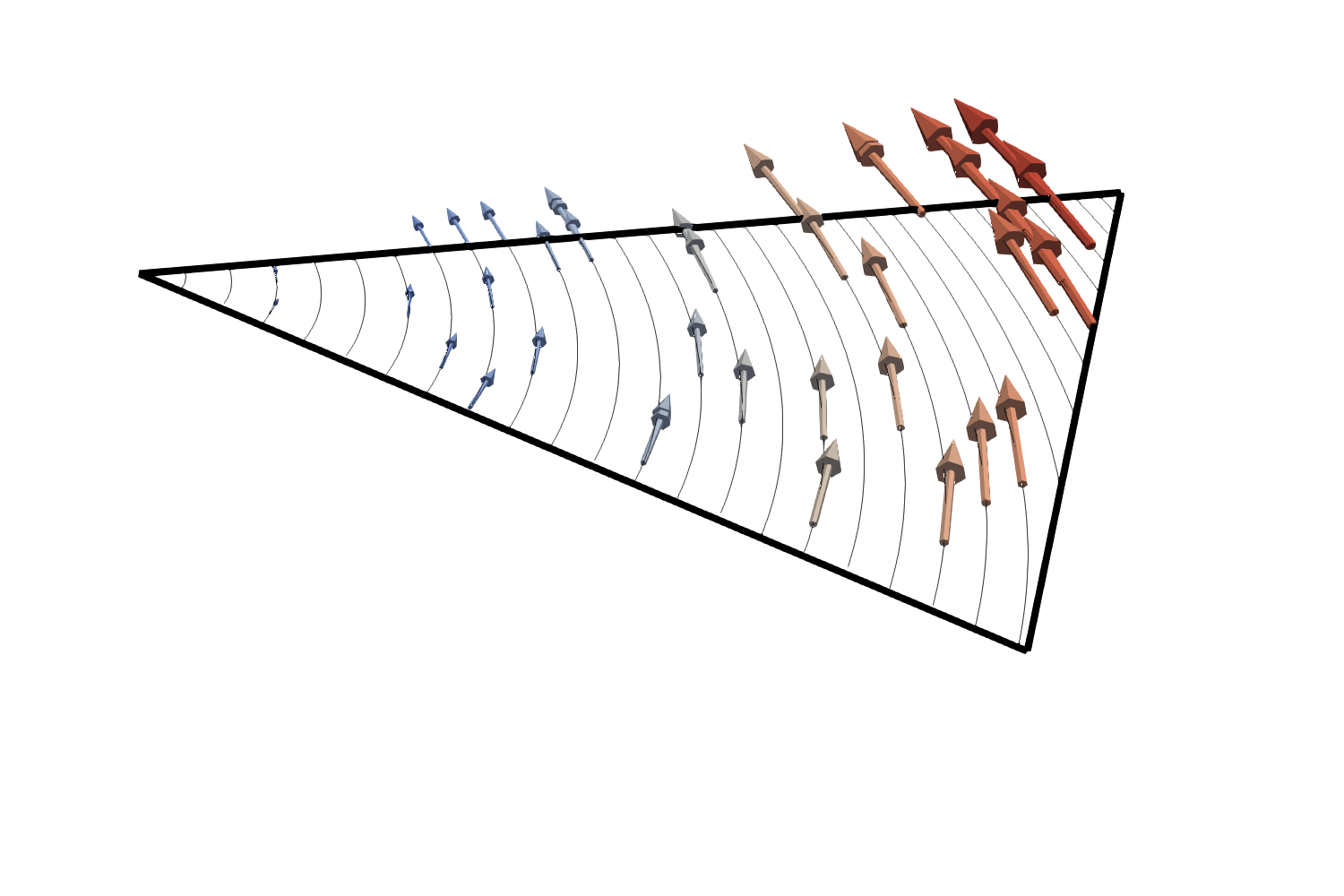}
        \hspace*{-0.0725\textwidth}
        \includegraphics[width=0.215\textwidth]{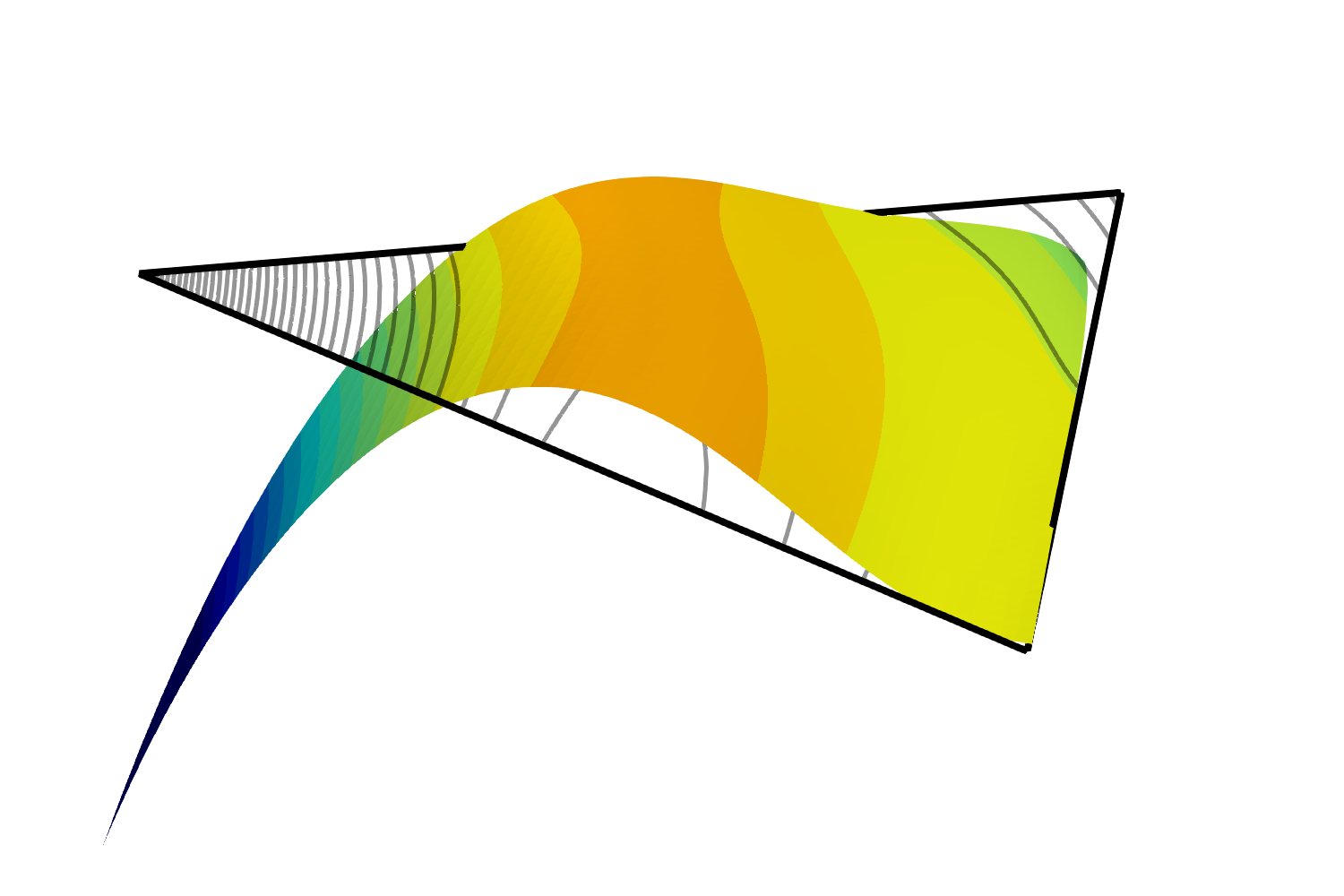}
        \hspace*{-0.0725\textwidth}
        \includegraphics[width=0.215\textwidth]{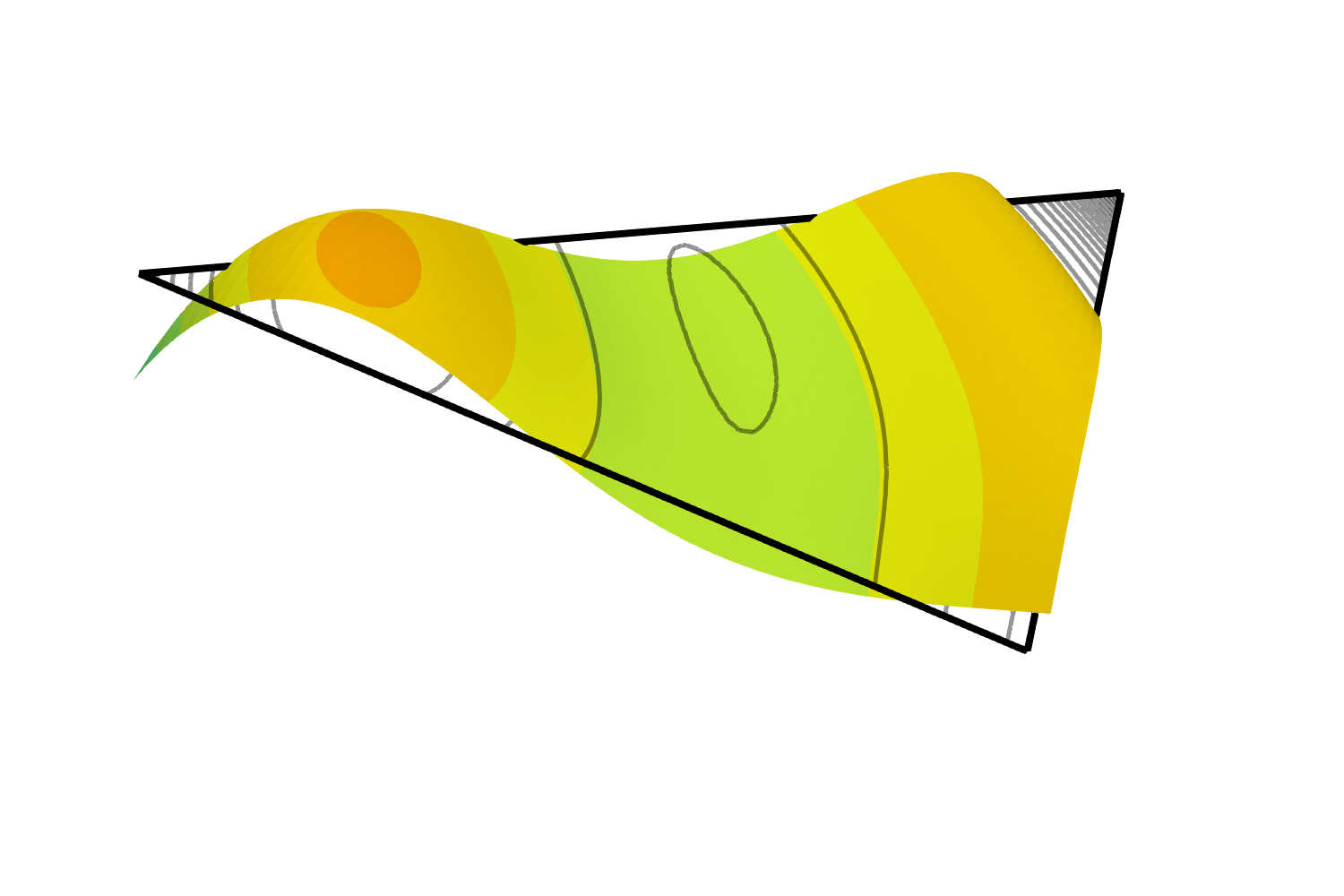}
        \hspace*{-0.0725\textwidth}
        \includegraphics[width=0.215\textwidth]{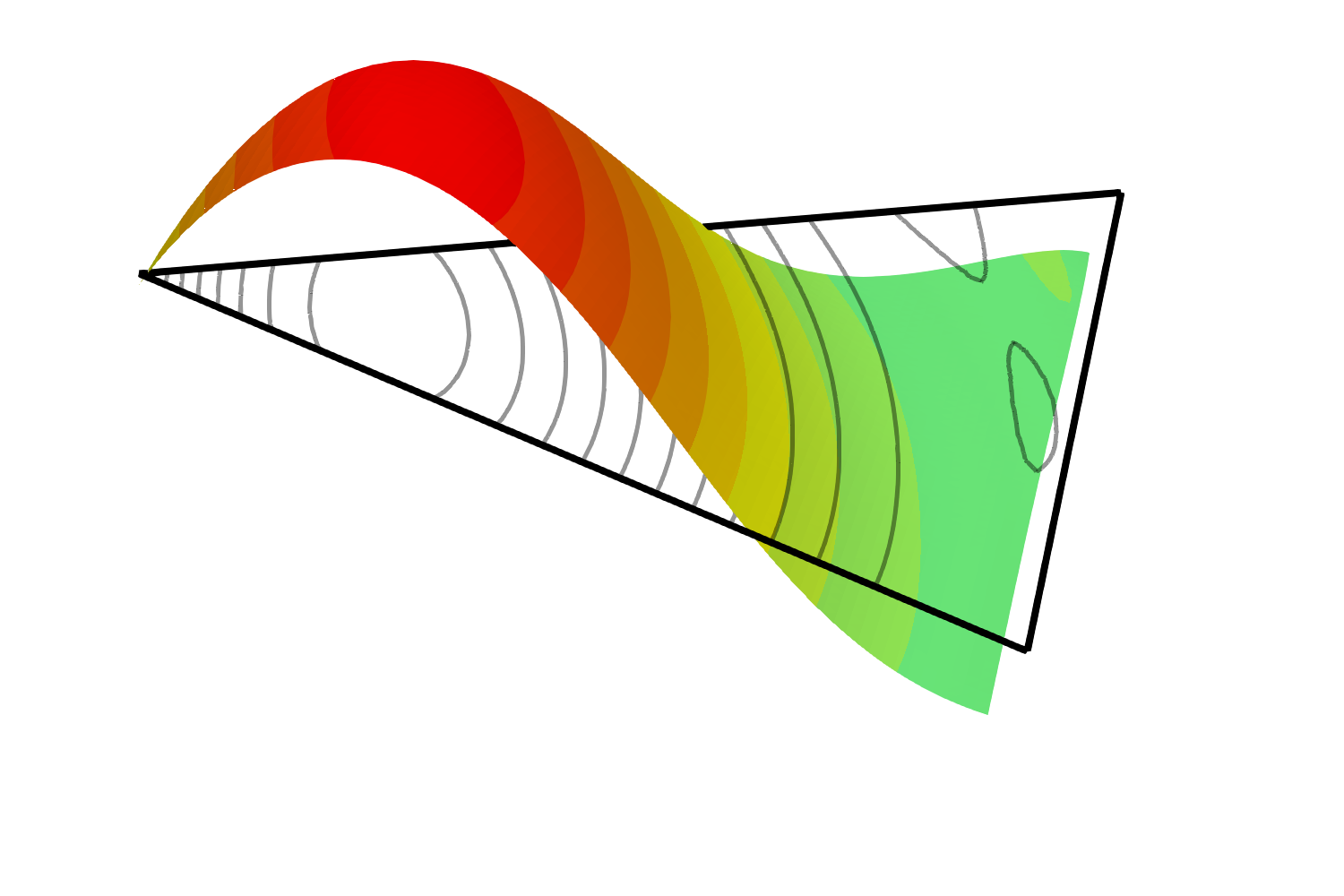}
        \hspace*{-0.0725\textwidth}
        \includegraphics[width=0.215\textwidth]{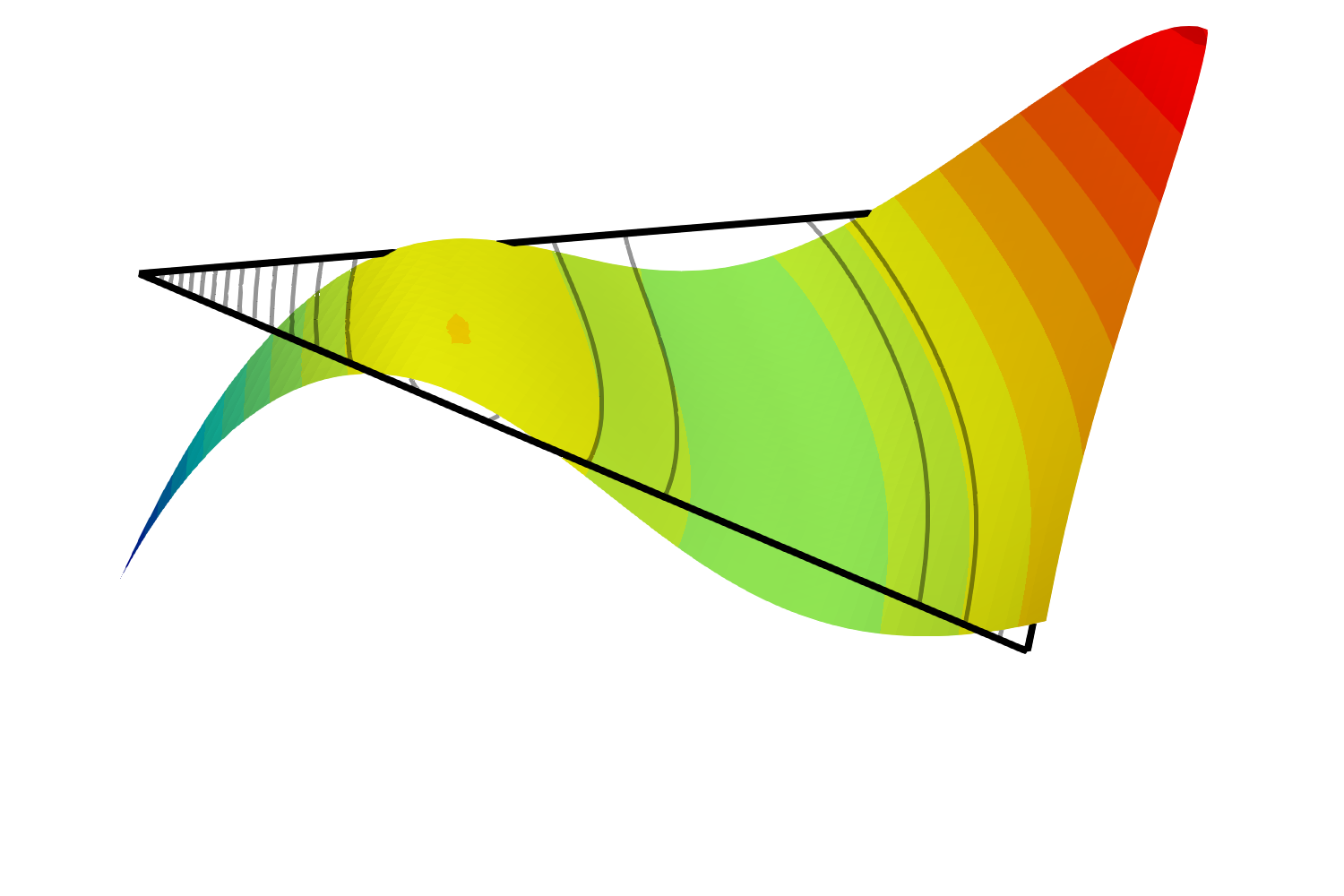}
        \hspace*{-0.0725\textwidth}
        \includegraphics[width=0.215\textwidth]{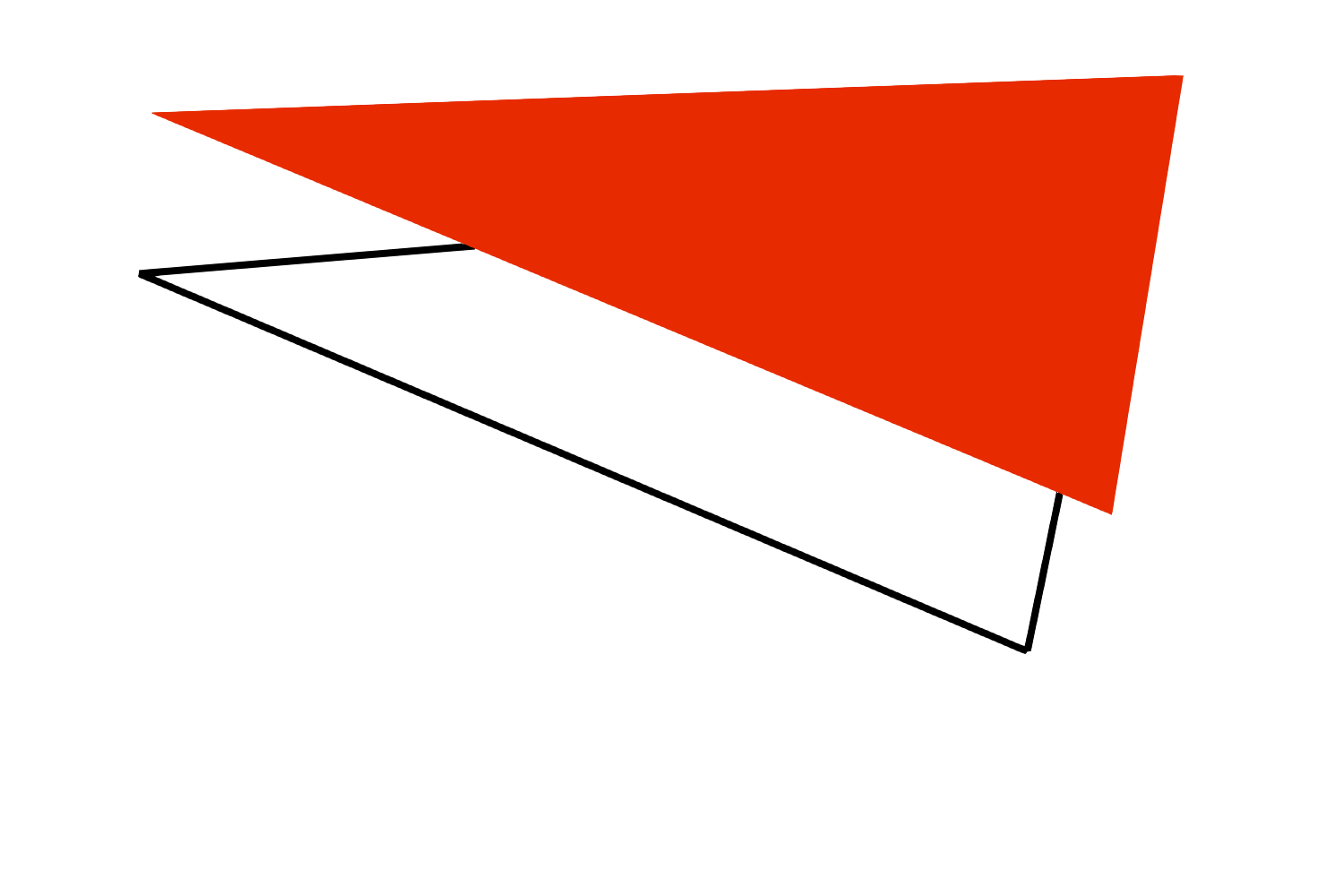}
        \hspace*{-0.02\textwidth}
    \end{center} \vspace*{-0.3cm} 
    \caption{(Non-constant) Flow field (left) and shape functions of the weak Trefftz space for $p=4$ for a triangle and $\calL =\bb\cdot\nabla$ as considered in \cref{sec:adv}.}
   \label{fig:advshapes}  
\end{figure}

For the numerical study we choose $\Omega = (0,1)^3$ and $\bb = (-\sin(x_2),\cos(x_1),x_1)^T$ and choose the r.h.s. data $f$ and $u_D$ so that $u =\sin(x_1)\sin(x_2)\sin(x_3)$ is the exact solution. Starting from a coarse simplicial and unstructured mesh with mesh size $h\approx 0.5$ we apply successive uniform refinements to compare the standard DG and the weak embedded Trefftz DG method.

\begin{figure}[ht]
    \resizebox{0.9\linewidth}{!}{
    \begin{tikzpicture} [spy using outlines={circle, magnification=4, size=2cm, connect spies}]
        \begin{groupplot}[%
          group style={%
            group name={my plots},
            group size=2 by 1,
            horizontal sep=1.5cm,
          },
        legend style={
            legend columns=6,
            at={(-0.2,-0.2)},
            anchor=north,
            draw=none
        },
        xlabel={$h$},
        ymajorgrids=true,
        grid style=dashed,
        cycle list name=paulcolors2,
        ]      
        \nextgroupplot[ymode=log,xmode=log,x dir=reverse, ylabel={$L^2$-error}]
            \foreach \p in {3,4,5}{
                \addplot+[discard if not={p}{\p}] table [x=h, y=error, col sep=comma] {results/adv3d24.csv}; 
                \addplot+[discard if not={p}{\p}] table [x=h, y=svdterror, col sep=comma] {results/adv3d24.csv};
            }
            \logLogSlopeTriangle{0.7}{0.1}{0.75}{4}{black}; 
            \logLogSlopeTriangle{0.7}{0.1}{0.56}{5}{black};
            \logLogSlopeTriangle{0.7}{0.1}{0.36}{6}{black}; 
        \nextgroupplot[ymode=log,xmode=log,x dir=reverse, ylabel={Time [s]}]
            \foreach \p in {3,4,5}{
                \addplot+[discard if not={p}{\p}] table [x=h, y=time, col sep=comma] {results/adv3d24.csv}; 
                \addplot+[discard if not={p}{\p}] table [x=h, y=svdttime, col sep=comma] {results/adv3d24.csv};
            }
            \legend{\DGmethod ($p=3$),\ETmethod ($p=3$),\DGmethod ($p=4$),\ETmethod ($p=4$),\DGmethod ($p=5$),\ETmethod ($p=5$)}
            \end{groupplot}
            \spy[purple,size=2cm] on (5.25,0.45) in node [fill=white] at (1.2,1);            
    \end{tikzpicture}}
   \vspace*{-0.3cm}  
   \caption{Comparison of $L^2$-error (left) and timings (right) for standard DG method and embedded Trefftz DG method for $p=3,4,5$ in 3D \rv{using 24 threads}.}
        \label{fig:adv}
    \end{figure}
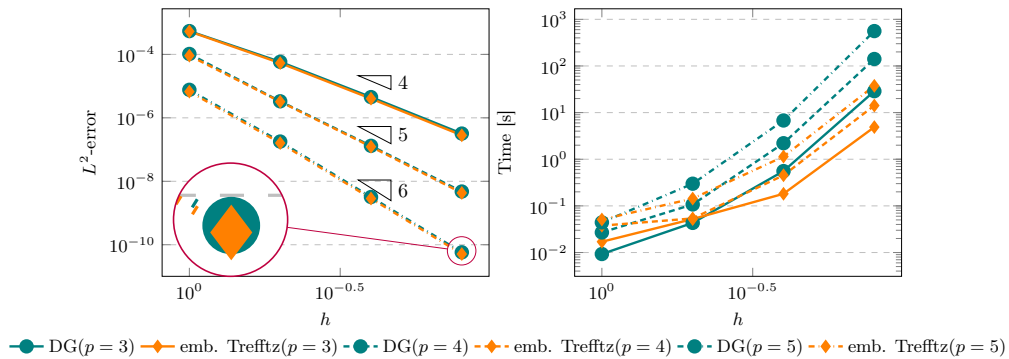
In the left half of \cref{fig:adv} we observe the convergence behavior of the embedded Trefftz DG method compared to the standard Upwind DG method for $p=3,4,5$. Both methods converge with optimal rate and we observe that there is only a marginal difference between the results. Moreover, a careful look at the numbers reveals that the solution of the embedded Trefftz DG method has a slightly smaller error on few occasions (see for instance the zoomed region in the left plot of \cref{fig:adv}). This may appear surprising at first glance as the approximation space is (by construction) smaller ($\IT^p(\Th) \subset\Vhp$ ), i.e. the approximation quality has not been improved. This suggests that the stability has been slightly improved in these cases. 
On the right half of \cref{fig:adv} we compare the runtime of the two approaches for $p=3,4,5$. We observe that on sufficiently fine meshes the costs associated to the two methods separate so that on the finest mesh level the computational costs differ already by more than an order of magnitude. Moreover, we observe that the embedded Trefftz DG method for $p=5$ is still much cheaper than the standard DG method for $p=4$ and only slightly more expensive than the standard DG method for $p=3$, i.e. we can see the step from DG to embedded Trefftz DG as a way to obtain also two orders of accuracy more for the same computation time, at least in the given example.

Although linear transport problems are not in the typical class of problems that are associated with Trefftz methods, the obtained results are very encouraging. Compared to the second order equations considered before the reduction of degrees of freedom is even higher and the gain in the computation time is remarkable.

\section{\rv{Comparison to Hybrid DG}} \label{sec:compareschemes}
We saw in the previous examples that the embedded Trefftz DG approach has two applications: First, it can be seen as an acceleration technique to reduce the costs of DG methods when solving linear systems. 
Second, the restriction of the DG space to a suitable subset can also have an effect on the stability of the method, cf. the Helmholtz problem in \cref{sec:helm}. For the former aspect alternative approaches exists. One technique that became very popular in the last decade is the class of Hybrid DG methods, see \cite{cockburn2016static,cockburn2009unified,egger2008hybrid}.
In the remainder of this section we want to shed
some light on the comparison between the two approaches.
To this end, we will first recap the structure of HDG method and make a rough conceptual comparison especially in terms of asymptotic complexity in \cref{sec:hdg:introduction} and afterwards 
compare the methods on the scalar examples from \cref{sec:poi} and \cref{sec:adv} and close the section with a rough conclusion. 

\subsection{Introduction} \label{sec:hdg:introduction}
In Hybrid DG methods additional unknowns are introduced on the element interfaces in order to allow for a decoupling of element unknowns. Element unknowns can then be removed from global linear systems by static condensation reducing the globally coupled \ndofs~(asymptotically for increasing $p$) from $\mathcal{O}(p^d)$ to   $\mathcal{O}(p^{d-1} )$ and thereby reducing the computational costs dramatically.

Let us sketch the structure of an HDG method. To this end, we restrict to a scalar PDE here, but extensions to the vector case are obvious. 
Let $F^p(\mathcal{F}_h)$ be the space of piecewise polynomials up to degree $p$ on each facet of
the mesh $\Th$ and $F_D^p(\mathcal{F}_h)$ and $F_0^p(\mathcal{F}_h)$ its subspaces with prescribed (inhomogeneous and homogeneous) values on Dirichlet-type boundaries. 
Then a typical discrete variational formulation of a Hybrid DG method is formulated in terms of the pair of volume and facet unknowns in the form:
Find $(u,u_F) \in V^p(\Th) \times F_D^p(\mathcal{F}_h)$ so that
\begin{equation} \label{eq:hdg}
a_h((u,u_F),(u,v_F)) =
\ell((v,v_F)) \text{ for all }(v,v_F) \in
V^p(\Th) \times F_0^p(\mathcal{F}_h).
\end{equation}
Here the bilinear form $a_h(\cdot,\cdot)$ can be very similar to a corresponding DG formulation with respect to the integral terms. However, direct couplings between volume unknowns of adjacent elements are avoided by involving the facet unknowns for the inter-element communication, so that the set of volume unknowns of one element is completely determined by the facet unknowns on the element boundary (and the r.h.s. data). This enables one of the main features of Hybrid DG schemes: the possibility to do \emph{static condensation}. 
This allows to eliminate the interior unknowns in $V^p(\Th)$ completely based on
the facet unknowns in $F^p(\mathcal{F}_h)$. This can be done on the level of the
variational formulation where with $u=u(u_F)$ one can formulate a discrete
formulation solely based on $u_F$ of the form: Find $u_F \in
F_D^p(\mathcal{F}_h)$ so that $a_h^\star(u_F,v_F) = \ell^\star(v_F)$ for
all $v_F \in F_0^p(\mathcal{F}_h)$. After solving for $u_F$ element-local problems can be solved to re-obtain $u(u_F) \in V^p(\Th)$. This elimination is often done merely on the linear algebra level based on a Schur complement strategy.
We note that in many Hybrid DG formulations for second order PDEs an auxiliary (flux) variable is also introduced locally which however can be eliminated alongside $u$ so that the structure of the global linear system for $u_F$ is not affected by this.

The asymptotic complexity of the Hybrid DG approach is the same as the one of the Embedded Trefftz DG method. The total \ndofs~is determined by the element unknowns, i.e. $\mathcal{O}(h^{-d}p^{d})$, the globally coupled \ndofs~ however only scales with $\mathcal{O}(h^{-d}p^{d-1})$. Similarly the local operations, i.e. the SVD or QR decomposition for the Embedded Trefftz or the Schur complement strategy require $\mathcal{O}(h^{-d} p^{3d})$ operations. Correspondingly, \nnzes and arithmetic operations for general purpose linear algebra solvers have the same complexity, cf. \cref{sec:complexity}.

We summarize that the Hybrid DG and the embedded Trefftz methods achieve the same asymptotic complexity, however the global unknowns are associated with element interfaces in the Hybrid DG case and associated with volume elements in the case of embedded Trefftz methods. In \cref{fig:schemecompare} we sketch the different $\dofs$ involved in the different methods for $p=3$.

In the first part of \cref{sec:algo} we have discussed two possible implementations of the embedded Trefftz-DG method. For the timings in this section we use the faster method, which avoids the assembly of the full DG matrix.

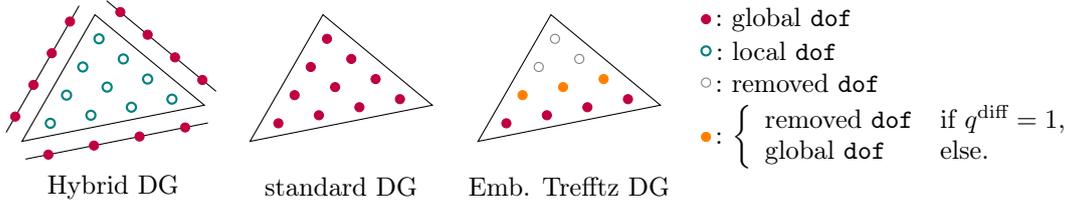
\begin{figure}\vspace*{-0.2cm}
  \begin{center}
    \begin{tikzpicture}[scale=0.6]
          \def\dirxx{1}
          \def\dirxy{0.2}
          \def\diryx{0.4}
          \def\diryy{0.7}

          \def\ax{0.2}
          \def\ay{0.2}

          \def\dd{0.4}

          \def\dirnxa{0.19611613513818404}
          \def\dirnya{-0.9805806756909202}

          \def\dirnxb{-0.8682431421244592}
          \def\dirnyb{0.49613893835683387}

          \def\dirnxc{0.6401843996644799}
          \def\dirnyc{0.7682212795973759}
          
          \draw (0,0) -- (4*\dirxx,4*\dirxy) -- (4*\diryx,4*\diryy) -- cycle;

          \draw (0+\dd*\dirnxa,0+\dd*\dirnya) -- (4*\dirxx+\dd*\dirnxa,4*\dirxy+\dd*\dirnya);

          \draw (0+\dd*\dirnxb,0+\dd*\dirnyb) -- (4*\diryx+\dd*\dirnxb,4*\diryy+\dd*\dirnyb);

          \draw (4*\diryx+\dd*\dirnxc,4*\diryy+\dd*\dirnyc) -- (4*\dirxx+\dd*\dirnxc,4*\dirxy+\dd*\dirnyc);

          \foreach \x/\y in {0.5/0, 1.5/0, 2.5/0, 3.5/0,
            0/1.33333, 1.33333/1.33333, 2.66666/1.33333,
            0/2.66666, 1.33333/2.66666,
            0/4}
          {
          }

          \foreach \x/\y in {0.44444/0.44444, 1.33333/0.44444, 2.22222/0.44444, 3.11111/0.44444,
            0.44444/1.33333, 1.33333/1.33333, 2.22222/1.33333,
            0.44444/2.22222, 1.33333/2.22222,
            0.44444/3.11111}
          {
            \draw[draw=teal,thick] (\x*\dirxx+\y*\diryx,\x*\dirxy+\y*\diryy) circle [radius=0.1cm] ;
          }

          \foreach \x/\y/\z in {0.5/0, 1.5/0, 2.5/0, 3.5/0}
          {
            \draw[draw=purple,fill=purple] (\x*\dirxx+\y*\diryx+\dd*\dirnxa,\x*\dirxy+\y*\diryy+\dd*\dirnya) circle [radius=0.1cm] ;
          }

          \foreach \x/\y/\z in {0/0.5, 0/1.5, 0/2.5, 0/3.5}
          {
            \draw[draw=purple,fill=purple] (\x*\dirxx+\y*\diryx+\dd*\dirnxb,\x*\dirxy+\y*\diryy+\dd*\dirnyb) circle [radius=0.1cm];
          }

          \foreach \x/\y in {3.5/0.5, 2.5/1.5, 1.5/2.5, 0.5/3.5}
          {
            \draw[draw=purple,fill=purple] (\x*\dirxx+\y*\diryx+\dd*\dirnxc,\x*\dirxy+\y*\diryy+\dd*\dirnyc) circle [radius=0.1cm] ;
          }

           \def\offset{5}

          \draw (\offset,0) -- (\offset+4*\dirxx,4*\dirxy) -- (\offset+4*\diryx,4*\diryy) -- cycle;

          \foreach \x/\y in {0.44444/0.44444, 1.33333/0.44444, 2.22222/0.44444, 3.11111/0.44444,
            0.44444/1.33333, 1.33333/1.33333, 2.22222/1.33333,
            0.44444/2.22222, 1.33333/2.22222,
            0.44444/3.11111}
          {
            \draw[draw=purple,fill=purple] (\offset+\x*\dirxx+\y*\diryx,\x*\dirxy+\y*\diryy) circle [radius=0.1cm] ;
          }

          \def\offset{10}

          \draw (\offset,0) -- (\offset+4*\dirxx,4*\dirxy) -- (\offset+4*\diryx,4*\diryy) -- cycle;

        \foreach \x/\y in {0.44444/0.44444, 1.33333/0.44444, 2.22222/0.44444, 3.11111/0.44444}
        {
            \draw[draw=purple,fill=purple] (\offset+\x*\dirxx+\y*\diryx,\x*\dirxy+\y*\diryy) circle [radius=0.1cm] ;
        }
        \foreach \x/\y in {
        0.44444/1.33333, 1.33333/1.33333, 2.22222/1.33333}
        {
            \draw[draw=orange,fill=orange] (\offset+\x*\dirxx+\y*\diryx,\x*\dirxy+\y*\diryy) circle [radius=0.1cm] ;
        }

          \foreach \x/\y in {
            0.44444/2.22222, 1.33333/2.22222,
            0.44444/3.11111}
          {
            \draw[draw=gray] (\offset+\x*\dirxx+\y*\diryx,\x*\dirxy+\y*\diryy) circle [radius=0.1cm] ;
          }
          \draw[draw=purple,fill=purple] (15,2.7) circle [radius=0.1cm] node[right] {:~global \texttt{dof}} ;
          \draw[draw=teal,thick] (15,2.0) circle [radius=0.1cm] node[right] {:~local \texttt{dof}} ;
          \draw[draw=gray] (15,1.3) circle [radius=0.1cm] node[right] {:~removed \texttt{dof}} ;
          \draw[draw=orange,fill=orange] (15,0.1) circle [radius=0.1cm] node[right] {:~$\left\{ \begin{array}{ll} \text{removed \texttt{dof}} & \text{if } q^{\text{diff}} = 1, \\ 
          \text{global \texttt{dof}} & \text{else.} \end{array} \right.$ };
          \node[](hdg) at (2,-1) {Hybrid DG};
          \node[](dg) at (7,-1) {standard DG};
          \node[](tdg) at (12,-1) {Emb. Trefftz DG};
        \end{tikzpicture}
  \end{center}
  \vspace*{-0.25cm}
  \caption{Sketch of global/local \dofs~ of Hybrid DG, standard DG and Embedded Trefftz DG for 2D simplex elements and polynomials degree $p=3$. Here $q^{\text{diff}}$ is the maximum differentiation index of $\mathcal{L}$.}
  \label{fig:schemecompare}
\end{figure}

\subsection{Poisson equation}
To compare the Hybrid DG method with the Embedded Trefftz DG method on a concrete example, we reconsider the example from \cref{sec:poi} and provide a simple hybrid DG discretization for it.
A hybrid DG analogue to the interior penalty formulation takes the form of \eqref{eq:hdg} with
\begin{align}\label{eq:hdglap}
  \begin{split}
    a_h((u,u_F),(u,v_F))
    &= \sum_{K \in \Th} \Big\{ \int_K \nabla u\nabla v\ dV -\int_{\partial K} \nabla u \cdot n_K (u-u_F)  + \nabla v \cdot n_K (v-v_F) \ dS \\
    & \qquad\qquad
        + \int_{\partial K}  \frac{\alpha p^2}{h}(u-u_F)(v-v_F)\ dS \Big\}\\
    \ell((v,v_F)) &=  \int_{\dom} f v \ dV 
  \end{split}
\end{align}
Note that there are no average and jump operators in \eqref{eq:hdglap} that directly involve neighboring element functions. Further, Dirichlet-type boundary conditions are directly imposed on the facet space which leads to the usual weak imposition of them on the element unknowns.
As for the DG formulation chosen in \cref{sec:poi} let us also mention that the Hybrid DG formulation is only one specific version (and a specifically simple one).

For second order elliptic and diffusion dominated problems, based on the similarity to (hybrid) mixed methods, element-local postprocessing schemes can be devised to obtain an additional order of accuracy for a postprocessed solution field. We neglect this aspect here but refer to \cite{LS_CMAME_2016,L_MTH_2010, oikawa2015hybridized, cockburn2016static} for more details.

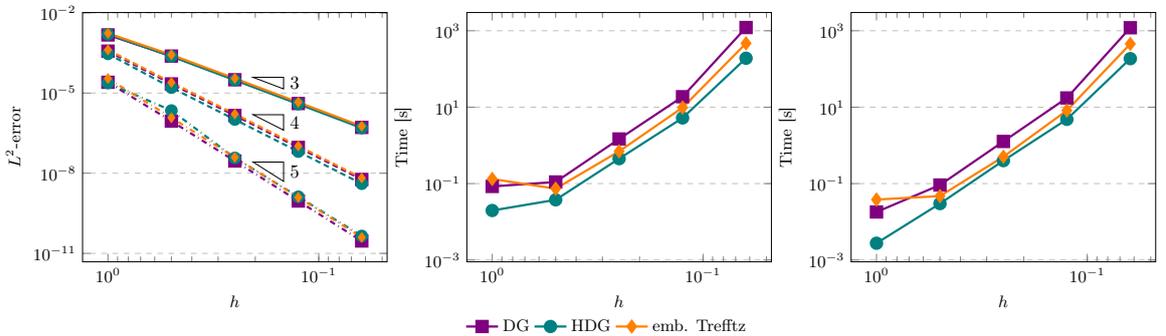
\begin{figure}[ht]
    \begin{center}
\hspace*{-0.02\textwidth}\resizebox{1.025\linewidth}{!}{
\begin{tikzpicture}
    \begin{groupplot}[%
      group style={%
        group size=3 by 1,
        horizontal sep=1.5cm,
      },
    legend style={
        legend columns=4,
        at={(-0.8,-0.2)},
        anchor=north,
        draw=none
    },
    ymajorgrids=true,
    grid style=dashed,
    ]      
    \nextgroupplot[xmode=log,ymode=log, xlabel={$h$}, x dir=reverse, ylabel=$L^2$-error, cycle list name=paulcolors]
    \foreach \p in {2,3,4}{
        \addplot+[discard if not={p}{\p}] table [x=h, y=l2error, col sep=comma] {results/poi3d12.csv};
        \addplot+[discard if not={p}{\p}] table [x=h, y=hdgerror, col sep=comma] {results/poi3d12.csv};
        \addplot+[discard if not={p}{\p}] table [x=h, y=svdt2error, col sep=comma] {results/poi3d12.csv};
    }
    \logLogSlopeTriangle{0.66}{0.1}{0.74}{3}{black}; 
    \logLogSlopeTriangle{0.66}{0.1}{0.59}{4}{black};
    \logLogSlopeTriangle{0.66}{0.1}{0.4}{5}{black};
    \nextgroupplot[xmode=log,ymode=log, xlabel={$h$}, x dir=reverse, ylabel={Time [s]}, cycle list name=paulcolors4,ymin=9*10^-4,ymax=3*10^3]
    \foreach \p in {4}{
      \addplot+[discard if not={p}{\p}] table [x=h, y=l2time, col sep=comma] {results/poi3d12.csv};
      \addplot+[discard if not={p}{\p}] table [x=h, y=hdgtime, col sep=comma] {results/poi3d12.csv};
      \addplot+[discard if not={p}{\p}] table [x=h, y=svdt2time, col sep=comma] {results/poi3d12.csv};
    }
    \nextgroupplot[xmode=log,ymode=log, xlabel={$h$}, x dir=reverse, ylabel={Time [s]}, cycle list name=paulcolors4,ymin=9*10^-4,ymax=3*10^3]
    \foreach \p in {4}{
      \addplot+[discard if not={p}{\p}] table [x=h, y=l2timeinv, col sep=comma] {results/poi3d12.csv};
      \addplot+[discard if not={p}{\p}] table [x=h, y=hdgtimeinv, col sep=comma] {results/poi3d12.csv};
      \addplot+[discard if not={p}{\p}] table [x=h, y=svdt2timeinv, col sep=comma] {results/poi3d12.csv};
    }
    \legend{\DGmethod, \HDGmethod, \ETmethod}
    \end{groupplot}
  \end{tikzpicture}}\hspace*{-0.005\textwidth}
\end{center}
\caption{Comparison between DG, Hybrid DG and embedded Trefftz DG method for the Poisson problem from \cref{sec:poi} in 3D. Left: $L^2$-error for $p \in \{2,3,4\}$ over successively refined meshes, center: computation time for $p=4$ with 12 threads, right: computation time for sparse matrix factorization $p=4$ with 12 threads.
}
\label{fig:poihdg}
\end{figure}

In \cref{fig:poihdg} we compare the computational results for the three methods, DG, Hybrid DG and Embedded Trefftz DG for fixed polynomial degrees and successively refined meshes. We observe that the differences in the error is negligible whereas the computation time of the Hybrid DG method is still smaller than the one of the Embedded Trefftz DG method.
Comparing the overall computation time with the time spend only to solve the global linear system we observe that this makes up for most of the time.

\subsection{Linear transport equation}
Let us now come to the first order example, the linear transport equation as in \cref{sec:adv}.
A Hybrid DG version of \eqref{eq:dgupw} takes the form of \eqref{eq:hdg} with 
\begin{align}\label{eq:hdgupw}
    a_h((u,u_F),(u,v_F))
    &= \sum_{K \in \Th} \Big\{ \int_K - u \bb \cdot \nabla v \ dV + \int_{\partial K } \bb\cdot\nx \hat{u} v~  dS
    + \int_{\partial K_{\text{out}} } \!\! \bb\cdot\nx (u_F - u) v_F ~  dS \Big\} \nonumber \\
    \ell((v,v_F)) &=  \int_{\dom} f v \ dV 
\end{align}
Here the upwind choice $\hat{u}$ is as follows: On inflow boundaries ($\bb \cdot \nx < 0$) we set $\hat{u} = u_F$, i.e. the facet value is taken whereas on outflow boundaries ($\bb \cdot \nx > 0$)  we set $\hat{u}=u$. To ensure that the facet value taken on inflow boundaries is meaningful, it is glued together with the outflow trace of the corresponding adjacent element by the stabilization term that include $v_F$. Note that testing with $v=0$ yields that $u_F$ is exactly the upwind trace at a corresponding facet. Hence, the hybrid Upwind DG formulation is equivalent to the Upwind DG formulation used above, cf. \cite{egger2008hybrid}. Note that inflow boundary conditions are prescribed through the facet space here and do not appear in the linear form.

\begin{figure}[ht]
    \begin{center}
\hspace*{-0.02\textwidth}\resizebox{1.025\linewidth}{!}{
\begin{tikzpicture}
    \begin{groupplot}[%
      group style={%
        group size=3 by 1,
        horizontal sep=1.5cm,
      },
    legend style={
        legend columns=3,
        at={(-0.8,-0.2)},
        anchor=north,
        draw=none
    },
    ymajorgrids=true,
    grid style=dashed,
    ]      
    \nextgroupplot[xmode=log,ymode=log, xlabel={$h$}, x dir=reverse, ylabel=$L^2$-error, cycle list name=paulcolors]
    \foreach \p in {2,3,4}{
        \addplot+[discard if not={p}{\p}] table [x=h, y=l2error, col sep=comma] {results/adv3d12.csv};
        \addplot+[discard if not={p}{\p}] table [x=h, y=hdgerror, col sep=comma] {results/adv3d12.csv};
        \addplot+[discard if not={p}{\p}] table [x=h, y=svdterror, col sep=comma] {results/adv3d12.csv};
    }
    \logLogSlopeTriangle{0.66}{0.1}{0.78}{3}{black}; 
    \logLogSlopeTriangle{0.66}{0.1}{0.58}{4}{black};
    \logLogSlopeTriangle{0.66}{0.1}{0.40}{5}{black};
    \nextgroupplot[xmode=log,ymode=log, xlabel={$h$}, x dir=reverse, ylabel={Time [s]}, cycle list name=paulcolors,]
    \foreach \p in {4}{
      \addplot+[discard if not={p}{\p}] table [x=h, y=l2time, col sep=comma] {results/adv3d12.csv};
      \addplot+[discard if not={p}{\p}] table [x=h, y=hdgtime, col sep=comma] {results/adv3d12.csv};
      \addplot+[discard if not={p}{\p}] table [x=h, y=svdttime, col sep=comma] {results/adv3d12.csv};
    }
    \nextgroupplot[xmode=log,ymode=log, xlabel={$h$}, x dir=reverse, ylabel={Time [s]}, cycle list name=paulcolors,]
    \foreach \p in {4}{
      \addplot+[discard if not={p}{\p}] table [x=h, y=l2timeinv, col sep=comma] {results/adv3d12.csv};
      \addplot+[discard if not={p}{\p}] table [x=h, y=hdgtimeinv, col sep=comma] {results/adv3d12.csv};
      \addplot+[discard if not={p}{\p}] table [x=h, y=svdttimeinv, col sep=comma] {results/adv3d12.csv};
    }
    \legend{\DGmethod, \HDGmethod, \ETmethod}
    \end{groupplot}
  \end{tikzpicture}}\hspace*{-0.005\textwidth}
\end{center}
\caption{Comparison between DG, Hybrid DG and embedded Trefftz DG method for the linear transport  problem from \cref{sec:adv} in 3D. Left: $L^2$-error for $p \in \{3,4,5\}$ over successively refined meshes, center: computation time for $p=5$ with 12 threads, right: computation time for sparse matrix factorization $p=5$ with 12 threads.
}
\label{fig:advhdg}
\end{figure}
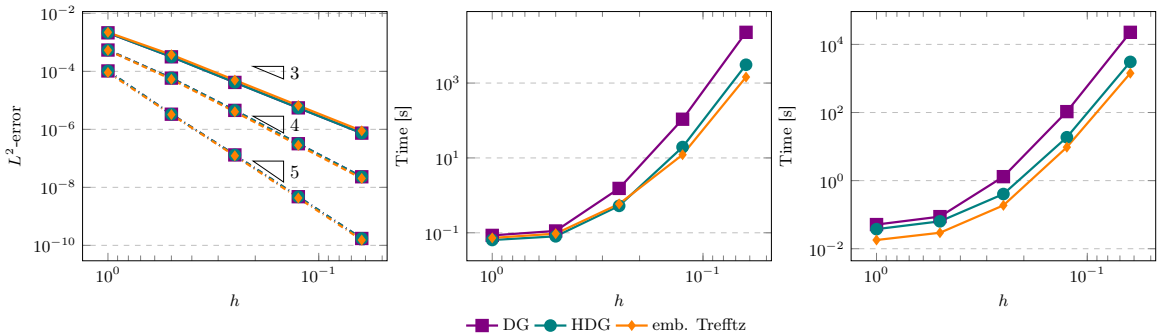

In \cref{fig:advhdg} we again compare DG, Hybrid DG and Embedded Trefftz DG methods. Again, the errors are comparable. In fact, DG and Hybrid DG are equivalent and hence yield exactly the same result. For the computation time we now observe that the Embedded Trefftz DG method is the fastest solution method as soon as the overhead in the system setup becomes negligible.

\subsection{Some number crunching}
In the previous two sections we considered first and second order scalar PDEs.
Next, we would like to compare the sparsity patterns of DG, HDG and Trefftz DG for a generic setting: 
We consider an unstructured simplicial mesh with 54 elements in 2D and 729 elements in 3D and compare different methods w.r.t. the \rv{number of degrees of freedom} (\ndofs)~and the \rv{number of non-zero entries} (\nnzes)~for $p=0,..,5$. We note that although \ndofs~is the simpler measure~\nnzes~has a more direct implication on the computational costs that are associated to solving linear systems.
In \cref{tab:hdg} A standard DG scheme, a corresponding HDG scheme and two embedded Trefftz DG schemes are compared. Here the embedded Trefftz DG schemes denoted as TDG(1) and TDG(2) are distinguished depending on the leading order of the differential operator involved. For first order problems a comparison of DG and HDG with TDG(1) is of interest, whereas TDG(2) is to be considered for comparison for second order formulations. For simplicity, we assume here that all element and facet unknowns of one element or facet couple with all unknowns of a corresponding facet or element (depending on the method) neighbor.

\begin{table}
    \begin{filecontents*}{results/compare_dofs.csv}
D,p,l2dofs,hdgdofs,svdtdofs1,svdtdofs2,l2nzes,hdgnzes,svdtnzes1,svdtnzes2
2,0,54,91,54,54,196,415,196,196
2,1,162,182,108,162,1764,1660,784,1764
2,2,324,273,162,270,7056,3735,1764,4900
2,3,540,364,216,378,19600,6640,3136,9604
2,4,810,455,270,486,44100,10375,4900,15876
2,5,1134,546,324,594,86436,14940,7056,23716
3,0,729,1612,729,729,3337,10360,3337,3337
3,1,2916,4836,2187,2916,53392,93240,30033,53392
3,2,7290,9672,4374,6561,333700,372960,120132,270297
3,3,14580,16120,7290,11664,1334800,1036000,333700,854272
3,4,25515,24180,10935,18225,4087825,2331000,750825,2085625
3,5,40824,33852,15309,26244,10464832,4568760,1471617,4324752
\end{filecontents*}
\pgfplotstabletypeset[
        col sep=comma, 
        columns={D,p,l2dofs,hdgdofs,svdtdofs1,svdtdofs2,l2nzes,hdgnzes,svdtnzes1,svdtnzes2},
        every nth row={2[+1]}{before row={\rowcolor[gray]{.875}}},
        every nth row={6}{before row=\midrule},
        columns/D/.style={column name=$d$},
        columns/p/.style={column name=$p$},
        columns/l2dofs/.style={int detect, column name=\ndofs~ DG,},
        columns/hdgdofs/.style={int detect, column name=HDG},
        columns/svdtdofs1/.style={int detect, column name=TDG(1)},
        columns/svdtdofs2/.style={int detect, column name=TDG(2)},
        columns/l2nzes/.style={int detect,  column name=\nnzes~ DG},
        columns/hdgnzes/.style={int detect, column name=HDG},
        columns/svdtnzes1/.style={int detect, column name=TDG(1)},
        columns/svdtnzes2/.style={int detect, column name=TDG(2)},
        column type=r,
        empty cells with={--}, 
        every head row/.style={before row=\toprule,after row=\midrule},
        every last row/.style={after row=\bottomrule}
     ]{results/compare_dofs.csv}
    
     \vspace*{0.5cm}\caption{Comparison of different measures (\ndofs, \nnzes) for the computational overhead for the solution of the largest global linear system related to the methods DG, Hybrid DG (after static condenstaion) and (embedded) Trefftz DG.}\vspace*{-0.5cm}  
     \label{tab:hdg}  
\end{table}  

We observe that for first order problems, the embedded Trefftz DG method even outperforms the HDG method. For second order problems HDG and embedded Trefftz DG methods are similar, in 2D HDG is cheaper, in 3D the embedded Trefftz DG method. However, optimized HDG formulations are able to reduce the polynomial degree on the element interfaces by one order in many cases without effecting the order of accuracy. In these cases the HDG method will still be cheaper. However, compared to the plain DG method the improvement of the embedded Trefftz DG method is already remarkable.

\section{Conclusion and possible extensions}\label{sec:outlook}
We have presented a method to reduce the matrix of a DG scheme post assembly, by projecting the polynomial basis onto Trefftz spaces. 
Several numerical examples have been presented, showing that the method matches the convergence rates of polynomial Trefftz spaces and showcasing its flexibility and potential in case of smooth coefficients and inhomogeneous equations.

While the Trefftz spaces used in the polynomial case are well understood, analytical properties of the space used for the embedding in the case of smooth coefficients and \rv{plane wave} Trefftz spaces will be addressed in a forthcoming paper.
One remarkable finding are the convincing results that have been obtained for the linear transport problem, a problem that is typically not associated with Trefftz methods.
Further, we saw that the method is more than an acceleration technique, as the restriction to a subspace can even improve stability as seen for the Helmholtz problem.
The flexibility of the approach suggests to consider the approach in many more and possibly more complex cases.

We focused on using the local operator and its kernel to emulate Trefftz spaces. 
However, the technique could also be applied to impose other constraints on the solution as an alternative to complicated constructions of special basis functions, weak impositions through the DG formulation or Lagrange multiplier techniques. 
Furthermore, the extraction done here for the kernel could similarly be applied for the range of certain differential operators or orthogonal (w.r.t. to a localizable inner product) complements.


One obvious drawback is the confinement to DG schemes.
It presents an alternative to other approaches that accelerate DG methods, such as HDG. 
However, using the method as a pure acceleration technique for DG methods may be less attractive, if sophisticated preconditioners and linear solvers for the DG discretization are available.
Moreover, the scaling of the computing time of the embedding construction in $p$ is not good, suggesting that really high polynomial degrees, e.g. beyond $p=10$, may not be feasible.
In these cases however, the method seems to at least offer a very convenient way to investigate Trefftz and Trefftz-type methods as a research tool. 

Another difficulty is, that combining the approach with (partially) conforming methods (such as those based on $H(\operatorname{div})$ or $H(\operatorname{curl})$ spaces) is not directly possible, \rv{due to the support of basis functions spanning multiple elements and overlapping partially the construction of a local embedding unclear. Additionally, the application of the PDE operator in strong form is only feasible element-wise.}

\phantomsection
\section*{Acknowledgments}
\addcontentsline{toc}{section}{Acknowledgments}
P. Stocker has been supported by the German Research Foundation (DFG) through grant 432680300 - SFB~1456.

\bibliographystyle{siam}
\bibliography{bib}

\begin{thebibliography}{10}

\bibitem{al2008method}
{\sc M.~J. Al-Khatib, K.~Grysa, and A.~Maci{\k{a}}g}, {\em The method of
  solving polynomials in the beam vibration problem}, J. Theoret. Appl. Mech.,
  46 (2008), pp.~347--366.

\bibitem{arnold2002unified}
{\sc D.~N. Arnold, F.~Brezzi, B.~Cockburn, and L.~D. Marini}, {\em Unified
  analysis of discontinuous {Galerkin} methods for elliptic problems}, {SIAM J.
  Numer. Anal.}, 39 (2002), pp.~1749--1779.

\bibitem{bgl2016}
{\sc L.~Banjai, E.~H. Georgoulis, and O.~Lijoka}, {\em A {T}refftz polynomial
  space-time discontinuous {G}alerkin method for the second order wave
  equation}, SIAM J. Numer. Anal., 55 (2017), pp.~63--86.

\bibitem{BMPS20}
{\sc P.~Bansal, A.~Moiola, I.~Perugia, and C.~Schwab}, {\em Space--time
  discontinuous {G}alerkin approximation of acoustic waves with point
  singularities}, IMA J. Numer. Anal., 41 (2021), pp.~2056--2109.

\bibitem{bcds20}
{\sc H.~Barucq, H.~Calandra, J.~Diaz, and E.~Shishenina}, {\em Space--time
  {T}refftz-{DG} approximation for elasto-acoustics}, Appl. Anal., 99 (2020),
  pp.~747--760.

\bibitem{buet2020trefftz}
{\sc C.~Buet, B.~Despr{\'e}s, and G.~Morel}, {\em Trefftz discontinuous
  {G}alerkin basis functions for a class of {F}riedrichs systems coming from
  linear transport}, Adv. Comput. Math., 46 (2020), pp.~1--27.

\bibitem{cessenat1998application}
{\sc O.~Cessenat and B.~Despr{\'e}s}, {\em Application of an ultra weak
  variational formulation of elliptic pdes to the two-dimensional helmholtz
  problem}, {SIAM J. Numer. Anal.}, 35 (1998), pp.~255--299.

\bibitem{cockburn2016static}
{\sc B.~Cockburn}, {\em Static condensation, hybridization, and the devising of
  the {HDG} methods}, in Building bridges: connections and challenges in modern
  approaches to numerical partial differential equations, Springer, 2016,
  pp.~129--177.

\bibitem{cockburn2009unified}
{\sc B.~Cockburn, J.~Gopalakrishnan, and R.~Lazarov}, {\em Unified
  hybridization of discontinuous {Galerkin}, mixed, and continuous {Galerkin}
  methods for second order elliptic problems}, {SIAM J. Numer. Anal.}, 47
  (2009), pp.~1319--1365.

\bibitem{10.1145/992200.992206}
{\sc T.~A. Davis}, {\em Algorithm 832: Umfpack v4.3---an unsymmetric-pattern
  multifrontal method}, ACM Trans. Math. Softw., 30 (2004), p.~196–199.

\bibitem{di2016review}
{\sc D.~A. Di~Pietro, A.~Ern, and S.~Lemaire}, {\em A review of hybrid
  high-order methods: formulations, computational aspects, comparison with
  other methods}, Building bridges: connections and challenges in modern
  approaches to numerical partial differential equations,  (2016),
  pp.~205--236.

\bibitem{EKSW15}
{\sc H.~Egger, F.~Kretzschmar, S.~M. Schnepp, and T.~Weiland}, {\em A
  {S}pace-{T}ime {D}iscontinuous {G}alerkin {T}refftz {M}ethod for {T}ime
  {D}ependent {M}axwell's {E}quations}, SIAM J. Sci. Comput., 37 (2015),
  pp.~B689--B711.

\bibitem{egger2008hybrid}
{\sc H.~Egger and J.~Sch{\"o}berl}, {\em A hybrid mixed discontinuous
  {Galerkin} method for convection-diffusion problems}, IMA J. Numer. Anal,
  (2008).

\bibitem{fure2020discontinuous}
{\sc H.~S. Fure, S.~Pernet, M.~Sirdey, and S.~Tordeux}, {\em A discontinuous
  galerkin trefftz type method for solving the two dimensional maxwell
  equations}, SN PDE, 1 (2020), pp.~1--25.

\bibitem{ghp09}
{\sc C.~J. {Gittelson}, R.~{Hiptmair}, and I.~{Perugia}}, {\em {Plane wave
  discontinuous Galerkin methods: Analysis of the $h$-version}}, {ESAIM, Math.
  Model. Numer. Anal.}, 43 (2009), pp.~297--331.

\bibitem{tdgschroedinger}
{\sc S.~Gómez and A.~Moiola}, {\em A space-time {T}refftz discontinuous
  {G}alerkin method for the linear {Schr\"odinger} equation}, arXiv preprint,
  arXiv:2106.04724,  (2021).

\bibitem{hmp11b}
{\sc R.~{Hiptmair}, A.~{Moiola}, and I.~{Perugia}}, {\em {Plane wave
  discontinuous Galerkin methods for the 2D Helmholtz equation: analysis of the
  $p$-version}}, {SIAM J. Numer. Anal.}, 49 (2011), pp.~264--284.

\bibitem{HMP11}
{\sc R.~Hiptmair, A.~Moiola, and I.~Perugia}, {\em Error analysis of
  {Trefftz-dis\-con\-tin\-u\-ous} {G}alerkin methods for the time-harmonic
  {M}axwell equations}, Math. Comp., 82 (2012), pp.~247--268.

\bibitem{TrefftzSurvey}
\leavevmode\vrule height 2pt depth -1.6pt width 23pt, {\em A survey of
  {T}refftz methods for the {H}elmholtz equation}, in Building Bridges:
  Connections and Challenges in Modern Approaches to Numerical Partial
  Differential Equations, G.~R. Barrenechea, F.~Brezzi, A.~Cangiani, and E.~H.
  Georgoulis, eds., Lect. Notes Comput. Sci. Eng., Springer, 2016.
\newblock pp. 237--278.

\bibitem{HMPS14}
{\sc R.~Hiptmair, A.~Moiola, I.~Perugia, and C.~Schwab}, {\em Approximation by
  harmonic polynomials in star-shaped domains and exponential convergence of
  {T}refftz {$hp$-dGFEM}}, ESAIM Math. Model. Num. Anal., 48 (2014),
  pp.~727--752.

\bibitem{Ho}
{\sc J.~{Horv\'ath}}, {\em {Basic sets of polynomial solutions for partial
  differential equations}}, {Proc. Am. Math. Soc.}, 9 (1958), pp.~569--575.

\bibitem{HuYuan18}
{\sc Q.~Hu and L.~Yuan}, {\em A plane wave method combined with local spectral
  elements for nonhomogeneous {H}elmholtz equation and time-harmonic {M}axwell
  equations}, Adv. Comput. Math., 44 (2018), pp.~245--275.

\bibitem{ImbertGeradDespres2014}
{\sc L.-M. Imbert-G{\'e}rard and B.~Despr{\'e}s}, {\em A generalized plane-wave
  numerical method for smooth nonconstant coefficients}, IMA J. Numer. Anal.,
  34 (2014), pp.~1072--1103.

\bibitem{qtrefftz}
{\sc L.-M. Imbert-Gérard, A.~Moiola, and P.~Stocker}, {\em A space-time
  quasi-{T}refftz {DG} method for the wave equation with piecewise-smooth
  coefficients}, arXiv preprint, arXiv:2011.04617,  (2021).

\bibitem{MiWi56}
{\sc E.~P. jun. {Miles} and E.~{Williams}}, {\em {The Cauchy problem for linear
  partial differential equations with restricted boundary conditions}}, {Can.
  J. Math.}, 8 (1956), pp.~426--431.

\bibitem{kk95}
{\sc E.~Kita and N.~Kamiya}, {\em Trefftz method: an overview}, Adv. Eng.
  Softw., 24 (1995), pp.~3--12.

\bibitem{kretzschmarphd}
{\sc F.~Kretzschmar}, {\em {The discontinuous Galerkin Trefftz method}}, PhD
  thesis, Technische Universit{\"a}t Darmstadt, 2015.

\bibitem{SpaceTimeTDG}
{\sc F.~Kretzschmar, A.~Moiola, I.~Perugia, and S.~M. Schnepp}, {\em {A priori
  error analysis of space-time Trefftz discontinuous Galerkin methods for wave
  problems}}, IMA J. Numer. Anal., 36 (2016), pp.~1599--1635.

\bibitem{KSTW2014}
{\sc F.~Kretzschmar, S.~M. Schnepp, I.~Tsukerman, and T.~Weiland}, {\em
  {Discontinuous Galerkin methods with Trefftz approximations}}, J. Comput.
  Appl. Math., 270 (2014), pp.~211--222.

\bibitem{L_MTH_2010}
{\sc C.~Lehrenfeld}, {\em {Hybrid Discontinuous Galerkin Methods for
  Incompressible Flow Problems}}, Master's thesis, RWTH Aachen, May 2010.

\bibitem{LS_CMAME_2016}
{\sc C.~Lehrenfeld and J.~Schöberl}, {\em High order exactly divergence-free
  hybrid discontinuous {Galerkin} methods for unsteady incompressible flows},
  Computer Methods in Applied Mechanics and Engineering, 307 (2016), pp.~339 --
  361.

\bibitem{LiShu2012}
{\sc F.~{Li}}, {\em {On the negative-order norm accuracy of a
  local-structure-preserving LDG method}}, {J. Sci. Comput.}, 51 (2012),
  pp.~213--223.

\bibitem{LiShu2006}
{\sc F.~{Li} and C.-W. {Shu}}, {\em {A local-structure-preserving local
  discontinuous Galerkin method for the Laplace equation}}, {Methods Appl.
  Anal.}, 13 (2006), pp.~215--234.

\bibitem{LLHC08}
{\sc Z.-C. Li, T.-T. Lu, H.-Y. Hu, and A.~H.-D. Cheng}, {\em Trefftz and
  collocation methods}, WIT Press, Southampton, 2008.

\bibitem{macikag2011trefftz}
{\sc A.~Maci{\k{a}}g}, {\em Trefftz functions for a plate vibration problem},
  J. Theoret. Appl. Mech., 49 (2011), pp.~97--116.

\bibitem{macig2005solution}
{\sc A.~Maci{\k{a}}g and J.~Wauer}, {\em Solution of the two-dimensional wave
  equation by using wave polynomials}, J. Engrg. Math., 51 (2005),
  pp.~339--350.

\bibitem{zbMATH06677366}
{\sc A.~{Maci\k{a}g}}, {\em {Solution of the three-dimension wave equation by
  using wave polynomials}}, {PAMM, Proc. Appl. Math. Mech.}, 4 (2004),
  pp.~706--707.

\bibitem{zbMATH02190564}
{\sc A.~{Maci\k{a}g} and J.~{Wauer}}, {\em {Solution of the two-dimensional
  wave equation by using wave polynomials}}, {J. Eng. Math.}, 51 (2005),
  pp.~339--350.

\bibitem{zbMATH06996162}
{\sc L.~Mascotto, I.~Perugia, and A.~Pichler}, {\em Non-conforming harmonic
  virtual element method: {{$h$}}- and {{$p$}}-versions}, J. Sci. Comput., 77
  (2018), pp.~1874--1908.

\bibitem{zbMATH07139237}
\leavevmode\vrule height 2pt depth -1.6pt width 23pt, {\em A nonconforming
  {Trefftz} virtual element method for the {Helmholtz} problem}, Math. Models
  Methods Appl. Sci., 29 (2019), pp.~1619--1656.

\bibitem{mps13}
{\sc J.~M. {Melenk}, A.~{Parsania}, and S.~{Sauter}}, {\em {General DG-methods
  for highly indefinite Helmholtz problems}}, {J. Sci. Comput.}, 57 (2013),
  pp.~536--581.

\bibitem{MiWi55a}
{\sc E.~P. Miles, Jr. and E.~Williams}, {\em A basic set of homogeneous
  harmonic polynomials in {$k$} variables}, Proc. Amer. Math. Soc., 6 (1955),
  pp.~191--194.

\bibitem{MiWi55b}
{\sc E.~P. Miles, Jr. and E.~Williams}, {\em A note on basic sets of
  homogeneous harmonic polynomials}, Proc. Amer. Math. Soc., 6 (1955),
  pp.~769--770.

\bibitem{AndreaPhD}
{\sc A.~Moiola}, {\em Trefftz-discontinuous {G}alerkin methods for
  time-harmonic wave problems}, PhD thesis, Seminar for applied mathematics,
  ETH Z\"urich, 2011.

\bibitem{moiola2013plane}
{\sc A.~Moiola}, {\em Plane wave approximation in linear elasticity}, Appl.
  Anal., 92 (2013), pp.~1299--1307.

\bibitem{mope18}
{\sc A.~Moiola and I.~Perugia}, {\em A space--time {T}refftz discontinuous
  {G}alerkin method for the acoustic wave equation in first-order formulation},
  Numer. Math., 138 (2018), pp.~389--435.

\bibitem{morel2018trefftz}
{\sc G.~Morel, C.~Buet, and B.~Despr{\'e}s}, {\em Trefftz discontinuous
  {G}alerkin method for {F}riedrichs systems with linear relaxation:
  application to the {P1} model}, Comput. Methods Appl. Math., 18 (2018),
  pp.~521--557.

\bibitem{oikawa2015hybridized}
{\sc I.~Oikawa}, {\em A hybridized discontinuous galerkin method with reduced
  stabilization}, Journal of Scientific Computing, 65 (2015), pp.~327--340.

\bibitem{zbMATH06596743}
{\sc I.~Perugia, P.~Pietra, and A.~Russo}, {\em A plane wave virtual element
  method for the {Helmholtz} problem}, ESAIM, Math. Model. Numer. Anal., 50
  (2016), pp.~783--808.

\bibitem{StockerSchoeberl}
{\sc I.~Perugia, J.~Sch\"{o}berl, P.~Stocker, and C.~Wintersteiger}, {\em Tent
  pitching and {T}refftz-{DG} method for the acoustic wave equation}, Comput.
  Math. Appl., 79 (2020), pp.~2987--3000.

\bibitem{PFT09}
{\sc S.~Petersen, C.~Farhat, and R.~Tezaur}, {\em A space-time discontinuous
  {G}alerkin method for the solution of the wave equation in the time domain},
  Internat. J. Numer. Methods Engrg., 78 (2009), pp.~275--295.

\bibitem{Qin05}
{\sc Q.-H. Qin}, {\em Trefftz finite element method and its applications},
  Appl. Mech. Rev., 58 (2005), pp.~316--337.

\bibitem{RoWi}
{\sc P.~C. {Rosenbloom} and D.~V. {Widder}}, {\em {Expansions in terms of heat
  polynomials and associated functions}}, {Trans. Am. Math. Soc.}, 92 (1959),
  pp.~220--266.

\bibitem{Stocker2022}
{\sc P.~Stocker}, {\em `ngstrefftz`: Add-on to ngsolve for trefftz methods},
  Journal of Open Source Software, 7 (2022), p.~4135.

\bibitem{trefftz1926}
{\sc E.~Trefftz}, {\em {Ein Gegenst{\"u}ck zum Ritzschen Verfahren}}, Proc. 2nd
  Int. Cong. Appl. Mech., Zurich, 1926,  (1926), pp.~131--137.

\bibitem{zbMATH02139293}
{\sc A.~{U\'sci{\l}owska-Gajda}, J.~A. {Ko{\l}odziej}, M.~{Cia{\l}kowski}, and
  A.~{Fr\k{a}ckowiak}}, {\em {Comparison of two types of Trefftz method for the
  solution of inhomogeneous elliptic problems}}, {Comput. Assist. Mech. Eng.
  Sci.}, 10 (2003), pp.~661--675.

\bibitem{yang2020trefftz}
{\sc J.~Yang, M.~Potier-Ferry, K.~Akpama, H.~Hu, Y.~Koutsawa, H.~Tian, and
  D.~S. Z{\'e}z{\'e}}, {\em Trefftz methods and {T}aylor series}, Arch. Comput.
  Methods Eng., 27 (2020), pp.~673--690.

\end{thebibliography}

\end{document}